\documentclass[a4paper,12pt]{article}

\usepackage {amsfonts, amssymb, amscd,amsthm, amsmath, eucal}

\usepackage{amsbsy}
\usepackage[all]{xy}

\theoremstyle{plain}
{
  \newtheorem{thm}{Theorem}[subsection]
  \newtheorem{defn}[thm]{Definition}
  \newtheorem{cor}[thm]{Corollary}
  \newtheorem{lem}[thm]{Lemma}
  \newtheorem{prop}[thm]{Proposition}
  \newtheorem{rem}[thm]{Remark}
  \newtheorem{clm}[thm]{Claim}
  \newtheorem{notation}[thm]{Notation}
  
  \newtheorem{constr}[thm]{Construction}
  \newtheorem{sssection}[thm]{}
}

{

}
\renewcommand{\subsubsection}{\sssection\rm}

\newcommand{\const}{\text{\rm const}}

\newcommand{\id}{id}
\newcommand{\Spec}{\text{Spec}}

\newcommand{\Ker}{\text{Ker}}

\newcommand{\Aff}{\mathbf {A}}
\newcommand{\Pro}{\mathbf {P}}

\newcommand \xra {\xrightarrow }

\newcommand \hra {\hookrightarrow }

\begin{document}

\title
{
On Grothendieck-Serre conjecture
concerning principal $G$-bundles over regular semi-local domains containing a finite field: II}

\author{Ivan Panin \footnote{The author acknowledges support of the
the RNF-grant 14-11-00456.}
}

\date{June 2, 2014}

\maketitle

\begin{abstract}
In three preprints
\cite{Pan1},
\cite{Pan3}
and the present one
we prove
Grothendieck-Serre's conjecture concerning principal $G$-bundles over regular semi-local domains $R$
containing {\bf a finite field}
(here $G$ is a reductive group scheme).
The preprint
\cite{Pan1} contains main geometric presentation theorems.
The present preprint
contains reduction of the Grothendieck--Serre's conjecture to the case
of semi-simple simply-connected group schemes
(see Theorem
\ref{MainThmGeometric}).
The preprint
\cite{Pan3}
contains a proof of
that conjecture for regular semi-local domains $R$
containing {\bf a finite field}.

The Grothendieck--Serre conjecture for the case of regular semi-local domains containing
{\bf an infinite field}
is proven in
\cite{FP}.
{\it Thus the conjecture holds for regular semi-local domains containing a field.}

The reduction is based on two purity results (Theorems
\ref{TheoremA}
and
\ref{PurityForSubgroup}).
The first of that purity results looks as follows.
Let $\mathcal O$ be a semi-local ring of finitely many closed points
on a $k$-smooth irreducible affine scheme, where $k$ is {\bf a finite field}.
Given a smooth $\mathcal O$-group scheme morphism
$\mu: G \to C$ of reductive
$\mathcal O$-group schemes, with a torus $C$
consider a functor
from $\mathcal O$-algebras to abelian groups
$S \mapsto {\cal F}(S):=C(S)/\mu(G(S))$.
Assuming additionally that the kernel
$ker(\mu)$ of $\mu$
is a reductive $\mathcal O$-group scheme,
we prove that this functor satisfies
a purity theorem for $\mathcal O$.
In the case of an infinite field $k$ the two purity results are proven in
\cite{Pa2}.

Examples to mentioned purity results are considered at the very end of the preprint.
\end{abstract}

\section{Introduction}
\label{SectIntroduction}
Recall
\cite[Exp.~XIX, Defn.2.7]{D-G}
that an $R$-group scheme $G$ is called reductive
(respectively, semi-simple; respectively, simple), if it is affine and smooth
as an $R$-scheme and if, moreover, for each ring homomorphism
$s:R\to\Omega(s)$ to an algebraically closed field $\Omega(s)$,
its scalar extension $G_{\Omega(s)}$
is connected and reductive (respectively, semi-simple; respectively, simple) algebraic
group over $\Omega(s)$.
{\it Stress that all the groups $G_{\Omega(s)}$ are connected}.
The class of reductive group schemes
contains the class of semi-simple group schemes which in turn
contains the class of simple group schemes. This notion of a
simple $R$-group scheme coincides with the notion of a {simple semi-simple
$R$-group scheme from Demazure---Grothendieck
\cite[Exp.~XIX, Defn.~2.7 and Exp.~XXIV, 5.3]{D-G}.} {\it
Throughout the paper $R$ denotes an integral noetherian domain and $G$
denotes a reductive $R$-group scheme, unless explicitly stated
otherwise}.
\par
A well-known conjecture due to J.-P.~Serre and A.~Grothendieck
\cite[Remarque, p.31]{Se},
\cite[Remarque 3, p.26-27]{Gr1},
and
\cite[Remarque 1.11.a]{Gr2}
asserts that given a regular local ring $R$ and its field of
fractions $K$ and given a reductive group scheme $G$ over $R$ the
map
$$ H^1_{\text{\'et}}(R,G)\to H^1_{\text{\'et}}(K,G), $$
\noindent
induced by the inclusion of $R$ into $K$, has trivial kernel.

\begin{thm}
\label{MainThmGeometric}
Let $k$ be {\bf a finite field}. Assume that for any irreducible $k$-smooth affine variety $X$ and any
finite family of its closed points $x_1, x_2,\dots, x_n$ and the semi-local $k$-algebra
$\mathcal O:= \mathcal O_{X,x_1, x_2,\dots, x_n}$
and all semi-simple simply connected reductive $\mathcal O$-group schemes $H$
the pointed set map
$$H^1_{\text{\rm\'et}}(\mathcal O, H) \to H^1_{\text{\rm\'et}}(k(X), H),$$
\noindent
induced by the inclusion of $\mathcal O$ into its fraction field $k(X)$, has trivial kernel.

Then for any regular semi-local domain $\mathcal O$ of the form
$\mathcal O_{X,x_1, x_2,\dots, x_n}$ above
and any reductive $\mathcal O$-group scheme $G$
the pointed set map
$$H^1_{\text{\rm\'et}}(\mathcal O, G) \to H^1_{\text{\rm\'et}}(K, G),$$
\noindent
induced by the inclusion of $\mathcal O$ into its fraction field $K$, has trivial kernel.
\end{thm}

To state our purity Theorems recall certain notions.
Let  ${\cal F}$ be  a covariant functor from the category of
commutative noetherian $R$-algebras to the category
of abelian groups. For any $R$-algebra $S$ which is a noetherian domain consider the sequence
$${\cal F}(S)\to {\cal F}(L)\to \bigoplus_{\mathfrak p}\mathcal F(L)/Im[{\cal F}(S_{\mathfrak p})\to \mathcal F(L)]$$
where $\mathfrak p$ runs over
the height $1$ primes of $S$ and
$L$ is the field of fractions of $S$.
{\it We say that}
{\it $\mathcal F$ satisfies purity for $S$ if this sequence --- which is
clearly a complex --- is exact}. The purity for $S$ is equivalent to the following
property
$$
\bigcap_{\textrm{ht} \mathfrak p=1} Im[{\cal F}(S_{\mathfrak p}) \to {\cal F}(L)]=
Im[{\cal F}(S) \to {\cal F}(L)].
$$

\begin{thm}[Theorem A]
\label{TheoremA}
\label{PurityGeometric}
Let $\mathcal O$ be a semi-local ring of finitely many closed points
on a $k$-smooth irreducible affine scheme $X$, where $k$ is {\bf a finite field}.
Let
$$\mu: G\to C$$
be a smooth $\mathcal O$-morphism of reductive
$\mathcal O$-group schemes, with a torus $C$. Set
$H= ker (\mu)$ and suppose additionally that
$H$ is a reductive $\mathcal O$-group scheme.
Then the functor
$$\mathcal F: S\mapsto C(S)/\mu(G(S))$$
defined on the category of $\mathcal O$-algebras
satisfies purity for $\mathcal O$ regarded as an
$\mathcal O$-algebra
via the identity map.
\end{thm}
Clearly, this Theorem is similar to Theorem A from
\cite{Pa2}. However the proof of this Theorem is much more involved since the base field is finite.

Another functor satisfying purity is described by the formulae
(\ref{AnotherFunctor}) in Section
\ref{SectionOneMorePurityTheorem}.
The respecting purity result is Theorem
\ref{PurityForSubgroup}.

After the pioneering articles
\cite{C-T/P/S}
and
\cite{R}
on purity theorems for algebraic groups,
various versions of purity theorems were proved in
\cite{C-TO},
\cite{PS},
\cite{Z},
\cite{Pa}.
The most general result in the so called constant case was given in
\cite[Exm.3.3]{Z}. This result follows now from our Theorem (A) by
 taking $G$ to be a $k$-rational reductive group, $C=\mathbb G_{m,k}$
and
$\mu: G \to \mathbb G_{m,k}$ a dominant $k$-group morphism. The papers
\cite{PS},
\cite{Z},
\cite{Pa}
contain results for the nonconstant case. However they only consider specific examples
of algebraic scheme morphisms $\mu: G \to C$.
In the case of an infinite field $k$ the two purity results are proven in
\cite{Pa2}.


Let us to point out that
we use transfers for the functor
$R \mapsto C(R)$, but we do not use at all any kind of norm principle
for the homomorphism $\mu: G \to C$.

The preprint is organized as follows. In Section
\ref{SectNorms}
we construct
norm maps following a method from
\cite[Sect.6]{SV}.
In Section
\ref{SectUnramifiedElements}
we discuss
unramified elements.
A key point here is Lemma
\ref{KeyUnramifiedness}.
In Section
\ref{SectSpecializationMaps}
we discuss specialization maps.
A key point here is Corollary
\ref{TwoSpecializations}.
In Section \ref{SecEquatingGroups}
an equating statement is formulated.
It is proved in the Appendix.
In Section
\ref{NiceTriples}
a convenient technical tool is recalled and Theorem
\ref{ThmEquatingGroups}
is proved.
In Section
\ref{SectElemNisnevichSquare}
we recall a geometric presentation Theorem from
\cite{Pan1}
(see Theorem \ref{ElementaryNisSquare} below).
In Section
\ref{SectProofOfTheoremA}
Theorem A
is proved.
In Section
\ref{SectProofofTheoremB}
a Theorem B is proved. It extends Theorem A to
the case of local regular domains (not to semi-local) containing a field.
In Section
\ref{SectionOneMorePurityTheorem}
we consider the functor
(\ref{AnotherFunctor})
and prove a purity Theorem
\ref{PurityForSubgroup}
for that functor.
In Section
\ref{SectionGr_SerreConj}
Theorem
\ref{MainThmGeometric}
is proved.
Finally in Section
\ref{ExamplesOfPurityResults}
we collect several examples illustrating two Purity Theorems.

\subparagraph{Acknowledgments}
The author thanks very much A.Suslin for his interest to the topic of the present preprint.

\section{Norms}
\label{SectNorms}
Let $k\subset K\subset L$ be field extensions and assume that $L$
is finite separable over $K$. Let $K^{sep}$ be a separable closure
of $K$ and
$$\sigma_i:K\to K^{sep},\enspace 1\leq i\leq n$$
the different embeddings of $K$ into $L$. As in \S1, let $C$ be a
commutative algebraic group scheme defined over $k$. One can define
a norm map
$${\mathcal N}_{L/K}:C(L)\to C(K)$$
by
$${\mathcal N}_{L/K}(\alpha)=\prod_i C(\sigma_i)(\alpha) \in C(K^{sep})^{{\mathcal
G}(K)} =C(K)\;.
$$
Following Suslin and Voevodsky
\cite[Sect.6]{SV} we
generalize this construction to finite flat ring extensions.
\smallskip
Let $p:X\to Y$ be a finite flat morphism of affine schemes.
Suppose that its rank is constant, equal to $d$. Denote by
$S^d(X/Y)$ the $d$-th symmetric power of $X$ over $Y$.

\begin{lem} There is a canonical section
$$\mathcal N_{X/Y}: Y\to S^d(X/Y)$$
which satisfies the following three properties:
\begin{itemize}
\item
[(i)] Base change: for any map $f:Y'\to Y$ of affine schemes, putting
$X'=X\times_Y Y'$ we have a commutative diagram
$$\xymatrix{Y'\ar[rr]^{\mathcal N_{X'/Y'}}\ar[d]_{f}&&S^d(X'/Y')\ar[d]^{S^d(Id_X\times
f)}\\
             Y\ar[rr]^{\mathcal N_{X/Y}}&&S^d(X/Y)}$$
\item[(ii)]
Additivity: If $f_1: X_1\to Y$ and $f_2:X_2\to Y$ are finite
flat morphisms of degree $d_1$ and $d_2$ respectively, then,
putting $X=X_1\coprod X_2$, $f=f_1\coprod f_2$ and $d=d_1+d_2$, we
have a commutative diagram
$$\xymatrix{S^{d_1}(X_1/Y)\times S^{d_2}(X_2/Y)\ar[rr]^(0.63)\sigma&&S^d(X/Y)\\
&&\\
&Y\ar[uul]^{\mathcal N_{X_1/Y}\times \mathcal N_{X_2/Y}}\ar[uur]_{\mathcal N_{X/Y}}& }$$
where $\sigma$ is the canonical imbedding.\\
\item[(iii)]
Normalization: If $X=Y$ and $p$ is the identity, then
$\mathcal N_{X/Y}$ is the identity.
\end{itemize}
\end{lem}

\begin{proof} We construct a map $\mathcal N_{X/Y}$ and check that it has the
desired properties. Let $B=k[X]$ and $A=k[Y]$, so that $B$ is a
locally free $A$-module of finite rank $d$. Let $B^{\otimes
d}=B\otimes_A B\otimes_A\cdots \otimes_A B$ be the $d$-fold tensor product
of $B$ over $A$. The permutation group ${\frak S}_d$ acts on
$B^{\otimes d}$ by permuting the factors.  Let $S^d_A(B)\subseteq
B^{\otimes d}$ be the $A$-algebra of all the ${\frak
S}_d$-invariant elements of $B^{\otimes d}$. We consider
$B^{\otimes d}$ as an $S^d_A(B)$-module through the inclusion
$S^d_A(B)\subseteq B^{\otimes d}$ of $A$-algebras. Let $I$ be the
kernel of the canonical  homomorphism $B^{\otimes d}\to
\bigwedge^d_A(B  )$ mapping $b_1\otimes\cdots\otimes b_d$ to
$b_1\wedge\cdots\wedge b_d$. It is well-known (and easily checked
locally on $A$) that $I$ is generated by all  the elements $x\in
B^{\otimes d}$ such that $\tau(x)=x$ for some transposition
$\tau$. If $s$ is in $S^d_A(B)$, then $\tau(sx)=\tau(s)\tau(x)=sx$,
hence $sx$ is in $S^d(B)$ too. In other words, $I$ is an
$S^d_A(B)$-submodule of $B^{\otimes d}$. The induced $S^d_A(B)$-module
structure on $\bigwedge^d_A(B)$ defines an $A$-algebra homomorphism
$$\varphi:S^d_A(B) \to End_A(\bigwedge^d_A(B)) \; .$$
Since $B$ is locally free of rank $d$ over $A$, $\bigwedge^d_A(B)$
is an invertible $A$-module and we can canonically identify
$End_A(\bigwedge^d_A(B))$ with $A$. Thus we have a map

$$\varphi:S^d_A(B)\to A$$
and we define
$$N_{X/Y}: Y\to S^d(X/Y)$$
as the morphism of $Y$-schemes induced by $\varphi$. The
verification of properties (i), (ii) and (iii) is straightforward.
\end{proof}

\smallskip
Let $k$ be {\bf a finite field}. Let $\mathcal O$ be the semi-local ring of finitely many {\bf closed} points
on a smooth affine irreducible $k$-variety $X$.
Let $C$ be an affine smooth commutative $\mathcal O$-group scheme,
Let $p:X\to Y$ be a finite flat morphism of affine $\mathcal O$-schemes
and $f:X\to C$ any $\mathcal O$-morphism. We define {\it the norm} $N_{X/Y}(f)$ of $f$ {\it as the composite map}
\begin{equation}
\label{NormMap}
Y \xra{N_{X/Y}}
S^d(X/Y) \to S^d_{\mathcal O}(X) \xra{S^d_{\mathcal O}(f)} S^d_{\mathcal O}(C)\xra{\times} C
\end{equation}
Here we write $"\times"$ for the group law on $C$.
The norm maps $N_{X/Y}$ satisfy the following conditions
\begin{itemize}
\item[(i')]
Base change: for any map $f:Y'\to Y$ of affine schemes, putting
$X'=X\times_Y Y'$ we have a commutative diagram
$$
\begin{CD}
C(X)                  @>{(id \times f)^{*}}>>            C(X^{\prime})      \\
@V{N_{X/Y}}VV @VV{N_{X^{\prime}/Y^{\prime}}}V    \\
C(Y)                  @>{f^{*}}>>            C(Y^{\prime})
\end{CD}
$$
\item[(ii')]
multiplicativity: if $X=X_1 \amalg X_2$ then the diagram commutes
$$
\begin{CD}
C(X)                  @>{(id \times f)^{*}}>>            C(X_1) \times C(X_2)      \\
@V{N_{X/Y}}VV                              @VV{N_{X_1/Y}N_{X_2/Y}}V    \\
C(Y)                  @>{id}>>            C(Y)
\end{CD}
$$
\item[(iii')]
normalization: if $X=Y$ and the map $X \to Y$ is the identity then $N_{X/Y}=id_{C(X)}$.
\end{itemize}


\section{Unramified elements}
\label{SectUnramifiedElements}
Let $k$ be {\bf a finite field}, $\mathcal O$ be the $k$-algebra from Theorem
\ref{TheoremA}
and $K$ be the fraction field of $\mathcal O$.
Let
$\mu: G\to C$
be the morphism of reductive $\mathcal O$-group schemes from Theorem
\ref{TheoremA}.
We work in this section with {\it the category of commutative $\mathcal O$-algebras}.
For a commutative $\mathcal O$-algebra $S$ set
\begin{equation}
\label{MainFunctor}
{\cal F}(S)=C(S)/\mu(G(S)).
\end{equation}
Let $S$ be an $\mathcal O$-algebra which is a domain and
let $L$ be its fraction field.
Define the {\it subgroup of $S$-unramified elements of $\mathcal F (L)$} as
\begin{equation}
\label{DefnUnramified}
\mathcal F_{nr,S}(L)=
\bigcap_{\mathfrak p \in Spec(S)^{(1)}} Im [ \mathcal F(S_{\mathfrak p})\to\mathcal F(L) ],
\end{equation}
where $Spec(S)^{(1)}$ is the set of hight $1$ prime ideals in $S$.
Obviously the image of $\mathcal F(S)$ in $\mathcal F(L)$ is contained in
$\mathcal F_{nr,S}(L)$. In most cases
$\mathcal F(S_{\mathfrak p})$
injects into
$\mathcal F(L)$
and
$\mathcal F_{nr,S}(L)$
is simply the intersection of all
$\mathcal F(S_{\mathfrak p})$.

For an element $\alpha \in C(S)$ we will write $\bar \alpha$ for its
image in ${\cal F}(S)$. {\it In this section we will write  ${\cal F}$
for the functor }
(\ref{MainFunctor}).
We will repeatedly use  the following result due to Nisnevich.

\begin{thm}[\cite{Ni}] 
\label{NisnevichCor}
Let $S$ be a $\mathcal O$-algebra which is discrete valuation ring with fraction field $L$.
The map
${\cal F}(S) \to {\cal F}(L)$
is injective.
\end{thm}

\begin{proof}
Let $H$ be the kernel of $\mu$.
Since $\mu$ is smooth and $C$ is a tori, the group scheme sequence
$$1 \to H \to G \to C \to 1$$
gives rise to a short exact sequence of group sheaves in the \'{e}tale topology.
In turn that sequence of sheaves induces a long exact sequence of pointed sets.
So, the boundary map
$\partial: C(S) \to \textrm{H}^1_{\text{\'et}}(S,H)$
fits in a commutative diagram
$$
\CD
C(S)/\mu (G(S)) @>>> C(L)/\mu (G(L)) \\
@VVV @VVV \\
\textrm{H}^1_{\text{\'et}}(S,H) @>>> \textrm{H}^1_{\text{\'et}}(L,H). \\
\endCD
$$
in which the vertical arrows have  trivial kernels.
The bottom arrow has  trivial kernel by a Theorem from
\cite{Ni}, since $H$ is a reductive $\mathcal O$-group scheme.
Thus the top arrow has  trivial kernel too.

\end{proof}

\begin{lem}
\label{BoundaryInj}
Let $\mu: G \to C$ be the above morphism of our reductive group schemes.
Let $H= \ker (\mu)$.
Then for an $\mathcal O$-algebra $L$, where $L$ is a field, the boundary map
$\partial: C(L)/{\mu (G(L))} \to \textrm{H}^1_{\text{\'et}}(L,H)$
is injective.
\end{lem}

\begin{proof}
For a $L$-rational point $t \in C$ set
$H_t = \mu^{-1}(t)$. The action by  left multiplication of $H$ on $H_t$
makes $H_t$ into a left principal homogeneous
$H$-space and moreover
$\partial (t) \in \textrm{H}^1_{\text{\'et}}(L,H)$
coincides with the isomorphism class of
$H_t$. Now suppose that $s,t \in C(L)$ are such that
$\partial (s)=\partial (t)$. This means that
$H_t$ and $H_s$ are
isomorphic as  principal homogeneous $H$-spaces.
We must check that for certain $g \in G(L)$ one has
$t=s \mu(g)$.

Let $L^{sep}$ be a separable closure of $L$. Let
$\psi: H_s \to H_t$
be an isomorphism of  left $H$-spaces. For any
$r \in H_s(L^{sep})$
and
$h \in H_s(L^{sep})$
one has
$$
(hr)^{-1}\psi (hr)= r^{-1}h^{-1}h \psi (r)= r^{-1} \psi (r).
$$
Thus for any
$\sigma \in Gal(L^{sep}/L)$
and any
$r \in H_s(L^{sep})$
one has
$$
r^{-1}\psi(r)=(r^{\sigma})^{-1}\psi (r^{\sigma})=(r^{-1}\psi(r))^{\sigma}
$$
which means that the point
$u=r^{-1}\psi(r)$ is a $Gal(L^{sep}/L)$-invariant point of
$G(L^{sep})$. So $u \in G(L)$.
The following relation shows that the $\psi$ coincides with the right multiplication
by $u$. In fact, for any $r \in H_s(L^{sep})$ one has
$\psi (r)= rr^{-1}\psi (r)= ru$. Since $\psi$ is the right multiplication
by $u$ one has $t=s\mu(u)$, which proves the lemma.

\end{proof}

Let $k$, $\mathcal O$ and $K$ be as above in this Section.
Let $\mathcal K$ be a field containing $K$ and
$x: \mathcal K^* \to \mathbb Z$
be a discrete valuation vanishing on $K$.
Let $A_x$ be the valuation ring of $x$. Clearly,
$\mathcal O \subset A_x$.
Let
$\hat A_x$ and $\hat {\mathcal K}_x$
be the completions of $A_x$ and $\mathcal K$ with respect to $x$.
Let
$i:\mathcal K \hookrightarrow \hat {\mathcal K}_x$
be the inclusion. By Theorem
\ref{NisnevichCor}
the map
${\cal F}(\hat A_x)\to {\cal F}(\hat{\mathcal K}_x)$
is injective. We will identify
${\cal F}(\hat A_x)$
with its image under this map. Set
$$
{\cal F}_x(\mathcal K)=i_*^{-1}({\cal F}(\hat A_x)).
$$

The inclusion
$A_x\hookrightarrow \mathcal K$
induces a map
$
{\cal F}(A_x) \to {\cal F}(\mathcal K)
$
which is injective by Lemma
\ref{NisnevichCor}.
So both groups
${\cal F}(A_x)$ and ${\cal F}_x(\mathcal K)$
are subgroups of
${\cal F}(\mathcal K)$.
The following lemma shows that
${\cal F}_x(\mathcal K)$
coincides with the subgroup of $
{\cal F}(\mathcal K)$
consisting of all elements {\it unramified} at $x$.

\begin{lem}
\label{TwoUnramified}
${\cal F}(A_x)={\cal F}_x(\mathcal K)$.
\end{lem}

\begin{proof}
We only have to check the inclusion
$
{\cal F}_x(\mathcal K) \subseteq {\cal F}(A_x)
$.
Let
$
a_x \in {\cal F}_x(\mathcal K)
$
be an element. It determines the elements
$
a \in {\cal F}(\mathcal K)
$
and
$
\hat a \in {\cal F}(\hat A_x)
$
which coincide when regarded as elements of
$
{\cal F}(\hat{\mathcal K}_x)
$.
We denote this common element in
$
{\cal F}(\hat{\mathcal K}_x)
$
 by
$
\hat a_x
$.
Let $H=\ker (\mu)$ and let
$\partial: C(-) \to \textrm{H}^1_{\text{\'et}}(-,H)$
be the boundary map.

Let
$
\xi=\partial(a)\in \textrm{H}^1_{\text{\'et}}(\mathcal K,H)
$,
$
\hat \xi=\partial(\hat a)\in \textrm{H}^1_{\text{\'et}}(\hat A_x,H)
$
and
$\hat \xi_x=\partial(\hat a_x)\in \textrm{H}^1_{\text{\'et}}(\hat {\mathcal K_x},H)$
Clearly,
$
\hat \xi
$
and
$
\xi
$
both coincide with
$
\hat \xi_x
$
when regarded as elements of
$
\textrm{H}^1_{\text{\'et}}(\hat {\mathcal K_x},H)
$.
Thus one can glue
$
\xi
$
and
$
\hat \xi
$
to get a
$
\xi_x \in \textrm{H}^1_{\text{\'et}}(A_x,H)
$
which maps to
$
\xi
$
under the map induced by the inclusion
$
A_x \hookrightarrow \mathcal K
$
and maps to
$
\hat \xi
$
under the map induced by the inclusion
$
A_x \hookrightarrow \hat A_x
$.

We now show that
$
\xi _x
$
has the form
$
\partial (a_x^\prime)
$
for a certain
$
a_x^\prime \in {\cal F}(A_x)
$.
In fact, observe that the image
$
\zeta
$
of
$
\xi
$
in
$
\textrm{H}^1_{\text{\'et}}(\mathcal K, G)
$
is  trivial.
By Theorem
\cite{Ni}
the map
$$
\textrm{H}^1_{\text{\'et}}(A_x, G) \to \textrm{H}^1_{\text{\'et}}(\mathcal K,G)
$$
has  trivial kernel. Therefore the image
$
\zeta_x
$
of
$
\xi_x
$
in
$
\textrm{H}^1_{\text{\'et}}(A_x, G)
$
is trivial. Thus there exists an element
$
a_x^\prime \in {\cal F}(A_x)
$
with
$
\partial(a_x^\prime)=\xi_x \in \textrm{H}^1_{\text{\'et}}(A_x, H).
$

We now prove that
$
a_x^\prime
$
coincides with
$
a_x
$
in
$
{\cal F}_x(\mathcal K)
$.
Since
$
{\cal F}(A_x)
$
and
$
{\cal F}_x(\mathcal K)
$
are both subgroups of
$
{\cal F}(\mathcal K)
$,
it suffices to show that
$
a_x^\prime
$
coincides with the element
$
a
$
in
$
{\cal F}( \mathcal K)
$.
By Lemma
\ref{BoundaryInj}
the map
\begin{equation}
\label{BoundaryInjFormulae}
{\cal F}(\mathcal K)\xrightarrow{\partial} \textrm{H}^1_{\text{\'et}}(\mathcal K, H)
\end{equation}
is \underbar {injective}. Thus it suffices to check that
$
\partial(a_x^\prime)=\partial(a)
$
in
$
\textrm{H}^1_{\text{\'et}}(\mathcal K, H)
$.
This is indeed the case because
$
\partial(a_x^\prime)=\xi_x
$
and
$
\partial(a)=\xi
$,
and
$
\xi_x
$
coincides with
$
\xi
$
when regarded over
$
\mathcal K
$.
We have proved that
$
a_x^\prime
$
coincides with
$
a_x
$
in
$
{\cal F}_x(\mathcal K)
$.
Thus the inclusion
$
{\cal F}_x(\mathcal K) \subseteq {\cal F}(A_x)
$
is proved, whence the lemma.

\end{proof}

Let $k$, $\mathcal O$ and $K$ be as above in this Section.
\begin{lem}
\label{KeyUnramifiedness}
Let $B \subset A$ be a finite extension of Dedekind $K$-algebras, which are domains.
Let $0 \neq f \in A$ be such that $A/fA$ is finite over $K$.
Let $h \in B\cap fA$ be such that the induced map
$B/hB\to A/fA$ is an isomorphism.
Suppose
$hA=fA\cdot J^{\prime\prime}$
for an ideal
$J^{\prime\prime} \subseteq A$
co-prime to the ideal $fA$.

Let $E$ and $F$ be the field of fractions of $B$ and $A$ respectively.
Let $\alpha \in C(A_f)$ be such that
$\bar \alpha \in {\cal F}(F)$
is $A$-unramified. Then, for
$\beta= N_{F/E}(\alpha)$,
the class
$\bar \beta \in {\cal F}(E)$
is $B$-unramified.
\end{lem}

\begin{proof}
The only primes at which $\bar \alpha$ could be ramified are those which divide
$hA$.
Let $\mathfrak p$ be one of them. Check that $\bar \alpha$ is
unramified at $\mathfrak p$.

To do this we consider all primes
$\mathfrak q_1, \mathfrak q_2, \dots, \mathfrak q_n$
in $A$ lying over
$\mathfrak p$.
Let
$\mathfrak q_1$
be the unique prime dividing $f$ and lying over
$\mathfrak p$.
Then
$$
A\otimes_B  \hat B_{\mathfrak p} = \hat A_{\mathfrak q_1} \times \prod_{i \neq 1} \hat A_{\mathfrak q_i}
$$
with
$\hat A_{\mathfrak q_1}=\hat B_{\mathfrak p}$.
If $F$, $E$ are the fields of fractions of $A$ and $B$
then
$$
F \otimes_{B} \hat B_{\mathfrak p}=\hat F_{\mathfrak q_1} \times \prod_{i \neq 1}  \hat F_{\mathfrak q_n}
$$
and
$\hat F_{\mathfrak q_1}=\hat E_{\mathfrak p}$.
We will write
$\hat F_i$ for $\hat F_{\mathfrak q_i}$
and
$\hat A_i$ for $\hat A_{\mathfrak q_i}$.
Let
$$\alpha \otimes 1= (\alpha_1,\dots, \alpha_n)
\in C(\hat F_1) \times \dots \times  C(\hat F_n).$$
Clearly for $i \geq 2$ one has
$\alpha_i \in C(\hat A_i)$
and
$\alpha_1=\mu(\gamma_1)\alpha^{\prime}_1$
with
$\alpha^{\prime}_1 \in C(\hat A_1)=C(\hat B_{\mathfrak p})$
and
$\gamma_1 \in G(\hat F_1)=G(\hat E_{\mathfrak p})$.
Now
$\beta \otimes 1 \in C(\hat E_{\mathfrak p})$
coincides with the product
$$
\alpha_1N_{\hat F_2/\hat E_{\mathfrak p}}(\alpha_2)\cdots N_{\hat F_n/\hat E_{\mathfrak p}}(\alpha_n)=
\mu(\gamma_1)
[\alpha^{\prime}_1N_{\hat F_2/\hat E_{\mathfrak p}}(\alpha_2)\cdots N_{\hat F_n/\hat E_{\mathfrak p}}(\alpha_n)].
$$
Thus
$\overline {\beta \otimes 1}=\bar \alpha^{\prime}_1\overline {N_{\hat F_2/\hat E_{\mathfrak p}}(\alpha_2)}
\cdots \overline {N_{\hat F_n/\hat E_{\mathfrak p}}(\beta_n)} \in {\cal F}(\hat B_{\mathfrak p})$.
Let
$i: E \hra \hat E_{\mathfrak p}$
be the inclusion and
$i_*:{\cal F}(E) \to {\cal F}(\hat E_{\mathfrak p})$
be the induced map.
Clearly
$i_*(\bar \beta)=\overline {\beta \otimes 1}$
in
${\cal F}(\hat E_{\mathfrak p})$.
Now
Lemma
\ref{TwoUnramified}
shows that
the element
$\bar \beta \in {\cal F}(E)$
belongs to
${\cal F}(B_{\mathfrak p})$.
Hence $\bar \beta$ is $B$-unramified.

\end{proof}

\section{Specialization maps}
\label{SectSpecializationMaps}
Let $k$ be {\bf a finite field}, $\mathcal O$ be the $k$-algebra from Theorem
\ref{TheoremA}
and $K$ be the fraction field of $\mathcal O$.
Let
$\mu: G\to C$
be the morphism of reductive $\mathcal O$-group schemes from Theorem
\ref{TheoremA}.
We work in this section with {\it the category of commutative $K$-algebras} and with
the functor
\begin{equation}
\label{MainFunctor2}
\mathcal F: S\mapsto C(S)/\mu(G(S))
\end{equation}
defined on the category of $K$-algebras. So, we assume in this Section that
each ring from this Section is equipped with a distinguished $K$-algebra structure and
each ring homomorphism from this Section respects that structures.
Let $S$ be an $K$-algebra which is a domain and
let $L$ be its fraction field.
Define the {\it subgroup of $S$-unramified elements $\mathcal F_{nr,S}(L)$ of $\mathcal F (L)$} by formulae
(\ref{DefnUnramified}).

For a regular domain $S$ with the fraction field $\cal K$
and each height one prime $\mathfrak p$ in $S$
we construct {\bf specialization maps}
$s_{\mathfrak p}: {\cal F}_{nr, S}({\cal K}) \to {\cal F} {(K(\mathfrak p))}$,
where $\cal K$ is the field of fractions of $S$ and
$K(\mathfrak p)$
is the residue field of
$R$ at the prime $\mathfrak p$.

\begin{defn}
\label{SpecializationDef}
Let
$Ev_{\mathfrak p}: C(S_{\mathfrak p}) \to C(K(\mathfrak p))$
and
$ev_{\mathfrak p}: {\cal F}(S_{\mathfrak p}) \to {\cal F}(K(\mathfrak p))$
be the maps induced by the canonical $K$-algebra homomorphism
$S_{\mathfrak p} \to K(\mathfrak p)$.
Define a homomorphism
$s_{\mathfrak p}: {\cal F}_{nr, S}({\cal K}) \to {\cal F} {(K(\mathfrak p))}$
by
$s_{\mathfrak p}(\alpha)= ev_{\mathfrak p}(\tilde \alpha)$,
where
$\tilde \alpha$
is a lift of $\alpha$ to
${\cal F}(S_{\mathfrak p})$.
Theorem
\ref{NisnevichCor}
shows that the map $s_{\mathfrak p}$ is well-defined.
It is called the specialization map. The map $ev_{\mathfrak p}$ is called the evaluation
map at the prime $\mathfrak p$.

Obviously for
$\alpha \in C(S_\mathfrak p)$
one has
$s_{\mathfrak p}(\bar \alpha)=\overline {Ev_{\mathfrak p}(\alpha)}
\in {\cal F}(K(\mathfrak p))$.
\end{defn}

\begin{lem}[\cite{H}]
\label{Harder}
Let $H^{\prime}$ be a smooth linear algebraic group over the field $K$.
Let $S$ be a $K$-algebra which is a Dedekind domain with field of fractions $\cal K$.
If $\xi \in \textrm{H}^1_{\text{\'et}}({\cal K},H^{\prime})$ is an $S$-unramified element for the functor
$\textrm{H}^1_{\text{\'et}}(-,H^{\prime})$
(see (\ref{DefnUnramified}) for the Definition), then $\xi$ can be lifted to an element of
$\textrm{H}^1_{\text{\'et}}(S,H^{\prime})$.
\end{lem}

\begin{proof}
Patching.

\end{proof}

\begin{thm}[\cite{C-TO}, Prop.2.2]
\label{C-TOthm}
Let $G^{\prime}=G_K$, where $G$ is the reductive $\mathcal O$-group scheme from this Section
(it is connected and even geometrically connected, since we follow
\cite[Exp.~XIX, Defn.2.7]{D-G}). Then
$$ker[\textrm{H}^1_{\text{\'et}}({K}[t],G^{\prime}) \to \textrm{H}^1_{\text{\'et}}({K}(t),G^{\prime})]=* \ .$$
\end{thm}

We need the following theorem.
\begin{thm}[Homotopy invariance]
\label{HomInvNonram}
Let $S \mapsto {\cal F}(S)$ be the functor defined by the formulae
(\ref{MainFunctor2})
and let
${\cal F}_{nr,K[t]}(K(t))$ be defined by the formulae
(\ref{DefnUnramified}).
Let
$K(t)$ be the rational function field in one variable.
Then one has
$$
{\cal F}(K)={\cal F}_{nr,K[t]}(K(t)).
$$
\end{thm}

\begin{proof}
The injectivity is clear, since the composition
$$
{\cal F}(K) \to {\cal F}_{nr,K[t]}(K(t)) \xrightarrow{s_0} {\cal F}(K)
$$
coincides with the identity (here
$
s_0
$
is the specialization map at the point zero defined in 4.6).

It remains to check the surjectivity. Let
$$\mu_K = \mu \otimes_{\mathcal O} K: G_K= G \otimes_{\mathcal O} K \to C \otimes_{\mathcal O} K = C_K .$$
Let
$a\in {\cal F}_{nr,K[t]}(K(t))$
and let
$H_K=ker(\mu_K)$. Since $\mu$ is smooth the $K$-group $H_K$ is smooth.
Since $G_K$ is reductive it is $K$-affine. Whence $H_K$ is $K$-affine.
Clearly,
the element
$
\partial (a)\in H_{et}^1(K(t),H_K)
$
is a class which for every closed point
$
x\in \Aff^1_K
$
belongs to the image of
$
H_{et}^1({\cal O}_x,H_K)
$.
Thus by Lemma
\ref{Harder},
$
\xi:= \partial (a)
$
can be represented by an element
$
\tilde \xi \in H_{et}^1(K[t],H_K)
$,
where
$
K[t]
$
is the polynomial ring.
Consider the diagram
$$
\SelectTips{cm}{}
\xymatrix @C=2pc @R=10pt {
{} & \tilde a \ar@{|->}[r] & {\tilde \xi} \ar@{|->}[r] & {\tilde \zeta} & {} \\
1 \ar[r] & {{\cal F}(K[t])} \ar[r]^{\partial\qquad} \ar[dd]_{\epsilon} & {H_{et}^1(K[t],H_K)} \ar[r] \ar[dd]_{\rho} & {H_{et}^1(K[t],G_K)} \ar[dd]_{\eta} \\
{} & {} & {} & {} & {} \\
1 \ar[r] & {{\cal F}(K(t))} \ar[r]^{\partial\qquad} & {H_{et}^1(K(t),H_K)} \ar[r] & {H_{et}^1(K(t),G_K)} \\
{} & {a} \ar@{|->}[r] & {\xi } \ar@{|->}[r] & {\ast} & {} }
$$
in which all the maps are canonical, the horizontal lines are exact sequences of pointed sets and
$ker(\eta)=*$ \
by Theorem
\ref{C-TOthm}.
Since
$
\xi
$
goes to the trivial element in
$
H_{et}^1(K(t),G_K)
$,
one concludes that
$
\eta(\tilde \zeta)= * .
$
Whence $\tilde \zeta = *$
by Theorem
\ref{C-TOthm}.
Thus there exists an element
$
\tilde a \in {\cal F}(K[t])
$
such that
$
\partial (\tilde a)= \tilde \xi.
$
The map
$
{\cal F}(K(t)) \to H_{et}^1(K(t),H_K)
$
is injective by Lemma
\ref{BoundaryInj}.
Thus
$
\epsilon(\tilde a)=a.
$
The map
$\mathcal F(K) \to \mathcal F(K[t])$
induced by the inclusion
$K \hra K[t]$
is surjective,
since the corresponding map
$C(K) \to C(K[t])$
is an isomorphism. Whence there exists an
$a_0 \in \mathcal F(K)$
such that its image in
$\mathcal F(K(t))$
coincides with the element
$a$.

\end{proof}

\begin{cor}
\label{TwoSpecializations}
Let $S \mapsto {\cal F}(S)$ be the functor defined in
(\ref{MainFunctor}).
Let
$$
s_0, s_1: {\cal F}_{nr, K[t]}(K(t)) \to {\cal F}(K)
$$
be the specialization maps at zero and at one
(at the primes (t) and (t-1)). Then $s_0=s_1$.
\end{cor}

\begin{proof}
It is an obvious consequence of Theorem
\ref{HomInvNonram}.
\end{proof}

\section{Equating group schemes}
\label{SecEquatingGroups} The following Theorem is a
straightforward analog of \cite[Prop.7.1]{OP}

\begin{thm}
\label{PropEquatingGroups}
Let $S$ be a regular semi-local
irreducible scheme such that the residue fields at all its closed points are {\bf finite}.
Let
$\mu_1: G_1\to C_1$ and
$\mu_2: G_2\to C_2$
be two smooth $S$-group scheme morphisms with tori $C_1$ and $C_2$.
Assume as well that $G_1$ and $G_2$ are reductive $S$-group schemes
which are forms of each other. Assume that $C_1$ and $C_2$ are forms of each other.
Let $T \subset S$ be a connected non-empty closed
sub-scheme of $S$, and
$\varphi: G_1|_T \to G_2|_T$,
$\psi: C_1|_T \to C_2|_T$
be $S$-group scheme isomorphisms such that
$(\mu_2|_{T}) \circ \varphi = \psi \circ (\mu_1|_{T})$.
Then there exists a finite \'{e}tale morphism
$\tilde S \xra{\pi} S$ together with its section $\delta: T \to
\tilde S$ over $T$ and $\tilde S$-group scheme isomorphisms
$\Phi: \pi^*{G_1} \to \pi^*{G_2}$
and
$\Psi: \pi^*{C_1} \to \pi^*{C_2}$
such that
\begin{itemize}
\item[(i)]
$\delta^*(\Phi)=\varphi$,
\item[(ii)]
$\delta^*(\Psi)=\psi$,
\item[(iii)]
$\pi^*(\mu_2) \circ \Phi = \Psi \circ \pi^*(\mu_1): \pi^*(G_1) \to \pi^*(C_2)$
\item[(iv)]
the scheme $\tilde S$ is irreducible.
\end{itemize}
\end{thm}
We refer to \cite[Prop.5.1]{PSV} for the proof of a slightly weaker
statement. The proof of the Theorem is done in the same
style with some additional technicalities (see Appendix).

\section{Nice triples}
\label{NiceTriples}
We study in the present Section certain packages of geometric data and morphisms of that
packages. The concept of "nice triples" is very closed to the one of "standard triples "
\cite[Defn.4.1]{Vo} and is inspired by the latter one. Let
$k$ be {\bf a finite field}, $X/k$ be a smooth geometrically irreducible affine variety,
$x_1,x_2, \dots ,x_n \in X$ be {\bf a family of its closed points}.
Let
$\mathcal O_{X, \{x_1,x_2,\dots,x_n\}}$
be the respecting semi-local ring.
\begin{defn}
\label{DefnNiceTriple}
Let
$U:= \text{Spec}(\mathcal O_{X, \{x_1,x_2,\dots,x_n\}})$.
A nice triple over $U$
consists of the following family of data:
\begin{itemize}
\item[(i)]
a smooth morphism
$q_U: \mathcal X \to U$, where $\mathcal X$ is an irreducible scheme,
\item[(iii)]
an element
$f \in \Gamma(\mathcal X, \mathcal O_{\mathcal X})$
\item[(ii)]
a section
$\Delta$
of the morphism
$q_U$.
\end{itemize}
These data must satisfy the following conditions:
\begin{itemize}
\item[(a)]
each component of each fibre of the morphism $q_U$ has dimension one,
\item[(b)]
the
$\Gamma(\mathcal X, \mathcal O_{\mathcal X})/f\cdot\Gamma(\mathcal X, \mathcal O_{\mathcal X})$
is a finite
$\Gamma(U, \mathcal O_{U})=O_{X, \{x_1,x_2,\dots,x_n\}}$-module,
\item[(c)]
there exists a finite surjective $U$-morphism
$\Pi: \mathcal X \to \Aff^1 \times U$,
\item[(d)]
$\Delta^*(f) \neq 0 \in \Gamma(U, \mathcal O_{U})$.
\end{itemize}

A morphism between two nice triples
$(q^{\prime}_U: \mathcal X^{\prime} \to U, f^{\prime}, \Delta^{\prime}) \to
(q: \mathcal X \to U, f, \Delta)$
over $U$
is an \'{e}tale morphism of $U$-schemes
$\theta: \mathcal X^{\prime} \to \mathcal X$
such that
\begin{itemize}
\item[(1)]
$q^{\prime}_U = q_U \circ \theta$,
\item[(2)]
$f^{\prime} = \theta^{*}(f)\cdot g^{\prime}$
for an element
$g^{\prime} \in \Gamma(\mathcal X^{\prime}, \mathcal O_{\mathcal X^{\prime}})$ \\
(in particular,
$\Gamma(\mathcal X^{\prime} , \mathcal O_{\mathcal X^{\prime} })/\theta^*(f)\cdot\Gamma(\mathcal X^{\prime} , \mathcal O_{\mathcal X^{\prime} })$
is a finite
$O_{X, \{x_1,x_2,\dots,x_n\}}$-module ),
\item[(3)]
$\Delta = \theta \circ \Delta^{\prime}$.
\end{itemize}
(Stress that there are no conditions concerning an interaction of $\Pi^{\prime}$ and $\Pi$ ).
\end{defn}

If
$U$ as in Definition
\ref{DefnNiceTriple}
then for any $U$-scheme $V$ and any closed point $u \in U$
set $V_u=u\times_U V$.
For a finite set $A$ denote $\sharp A$ the cardinality of $A$.

\begin{thm}
\label{ThmEquatingGroups}
Let $U$ be as in Definition
\ref{DefnNiceTriple}.
Let
$(q: \mathcal X \to U, f, \Delta)$
be a nice triple over $U$.
Let $G_{\mathcal X}$ be a reductive
$\mathcal X$-group scheme and
$G_U:= \Delta^*(G_{\mathcal X})$
and
$G_{\text{const}}$
be the pull-back of $G_U$ to
$\mathcal X$.
Let $C_{\mathcal X}$ be an
$\mathcal X$-tori and
$C_U:= \Delta^*(C_{\mathcal X})$
and
$C_{\text{const}}$
be the pull-back of $C_U$ to
$\mathcal X$.
Let
$\mu_{\mathcal X}: G_{\mathcal X} \to C_{\mathcal X}$
be an $\mathcal X$-group scheme morphism smooth as a scheme morphism.
Let
$\mu_U = \Delta^*(\mu_{\mathcal X})$
and
$\mu_{\const}: G_{\const} \to C_{\const}$
be the the pull-back of $\mu_U$ to
$\mathcal X$.

Then
there exist a morphism
$\theta: (q^{\prime}: \mathcal X^{\prime} \to U, f^{\prime}, \Delta^{\prime}) \to
(q: \mathcal X \to U, f, \Delta)$
between nice triples over $U$ and isomorphisms
$$\Phi: \theta^*(G_{\text{const}}) \to \theta^*(G_{\mathcal X}), \ \Psi: \theta^*(C_{\text{const}}) \to \theta^*(C_{\mathcal X})$$
of
$\mathcal X^{\prime}$-group schemes such that
\begin{itemize}
\item[\rm{(1)}]
$(\Delta^{\prime})^*(\Phi)= id_{G_U}$,
$(\Delta^{\prime})^*(\Phi)= id_{G_U}$
and
\begin{equation}
\label{PhiAndPsiRelation}
\theta^*(\mu_{\mathcal X}) \circ \Phi = \Psi \circ \theta^*(\mu_{const})
\end{equation}
\item[\rm{(2)}]
for the closed sub-scheme $\mathcal Z^{\prime}$ of $\mathcal X^{\prime}$ defined by $\{f^{\prime}=0\}$
and any closed point $u \in U$ the point $\Delta^{\prime}(u) \in \mathcal Z^{\prime}_u$ is the only
$k(u)$-rational point of $\mathcal Z^{\prime}_u$,
\item[\rm{(3)}]
for any closed point $u \in U$ and any integer $r \geq 1$ and for $\mathcal Z^{\prime}$ as in \rm{(2)} one has
$$\sharp\{z \in \mathcal Z^{\prime}_u| \text{deg}[z:u]=r \} \leq \ \sharp\{x \in \Aff^1_u| \text{deg}[z:u]=r \}$$
\end{itemize}
\end{thm}

\begin{proof}[Proof of Theorem \ref{ThmEquatingGroups}]
We can start by
almost literally repeating arguments from the proof of
\cite[Lemma
8.1]{OP1}, which involve the following purely geometric lemma
\cite[Lemma 8.2]{OP1}.
\par
For reader's convenience below we state that Lemma adapting
notation to the ones of Section \ref{NiceTriples}.
Namely, let $U$ be as in
the Theorem and
$(q: \mathcal X \to U,f,\Delta)$ be the nice triple over $U$.
By the definition
of a nice triple there exists a finite surjective morphism
$\Pi:\mathcal X\to\Aff^1\times U$ of $U$-schemes.
\begin{lem}
\label{Lemma_8_2} Let $\mathcal Y$ be a closed nonempty sub-scheme
of $\mathcal X$, finite over $U$. Let $\mathcal V$ be an open
subset of $\mathcal X$ containing $\Pi^{-1}(\Pi(\mathcal Y))$. There
exists an open set $\mathcal W \subseteq \mathcal V$ still
containing $q_U^{-1}(q_U(\mathcal Y))$ and endowed with a finite
surjective morphism
$\Pi^*: \mathcal W\to\Aff^1\times U$ {\rm(}in general
$\neq\Pi${\rm)}.
\end{lem}
Let $\Pi:\mathcal X\to\Aff^1\times U$ be the above finite
surjective $U$-morphism.
The following diagram summarises the situation:
$$ \xymatrix{
{}&\mathcal Z \ar[d]&{}\\
{\mathcal X - \mathcal Z\ }\ar@{^{(}->}@<-2pt>[r]&\mathcal X\ar@<2pt>[d]^{q_U}\ar[r]^>>>>{\Pi}& \Aff^1 \times U\\
{}& U\ar@<2pt>[u]^{\Delta}&{} } $$
\noindent
Here $\mathcal Z$ is the closed sub-scheme defined by the equation
$f=0$. By assumption, $\mathcal Z$ is finite over $U$. Let
$\mathcal Y=\Pi^{-1}(\Pi(\mathcal Z_{red} \cup \Delta(U)))$. Since
$\mathcal Z$ and $\Delta(U)$ are both finite over $U$ and since
$\Pi$ is a finite morphism of $U$-schemes, $\mathcal Y$ is also
finite over $U$. Denote by $y_1,\dots,y_m$ its closed points and
let $S=\text{Spec}(\mathcal O_{\mathcal X,y_1,\dots,y_m})$. Set
$T=\Delta(U)\subseteq S$.
Further, let
$G_{\mathcal X}$,
$G_U$,
$G_{\const}$,
$C_{\mathcal X}$,
$C_U$,
$C_{\const}$,
$\mu_{\mathcal X}$,
$\mu_U$,
$\mu_{\const}$
be as in the hypotheses of Theorem
\ref{ThmEquatingGroups}.
Finally,
let
$\varphi:G_{\const}|_T \to G_{\mathcal X}|_T$
and
$\psi:C_{\const}|_T \to C_{\mathcal X}|_T$
be the canonical
isomorphisms.
Clearly,
$(\mu_{\mathcal X}|_{T}) \circ \varphi = \psi \circ (\mu_{\const}|_{T})$
Recall that by assumption $\mathcal X$ is $U$-smooth and irreducible,
and thus $S$ is regular and irreducible.
\par
By Theorem
\ref{PropEquatingGroups}
there exists a finite
\'etale morphism $\theta_0:S^{\prime}\to S$, a section
$\delta:T\to S^{\prime}$ of $\theta_0$ over $T$
and isomorphisms
$$ \Phi_0:\theta^*_0(G_{\const,S})\to\theta^*_0(G_{\mathcal X}|_S) \ \text{and} \
\Psi_0:\theta^*_0(C_{\const,S})\to\theta^*_0(C_{\mathcal X}|_S)$$
\noindent
such that
$S^{\prime}$ is irreducible,
$\delta^*\Phi_0=\varphi$,
$\delta^*\Psi_0=\psi$
and
$$\theta^*_0(\mu_{\mathcal X}) \circ \Phi_0 = \Psi_0 \circ \theta^*_0(\mu_{\const}): \theta^*_0(G_{\const}) \to \theta^*_0(C_{\mathcal X}).$$

Let $p=q_U \circ \theta_0: S^{\prime} \to U$.
By Lemma \ref{SmallAmountOfPoints}
there exists a finite \'{e}tale morphism
$\rho: S^{\prime\prime} \to S^{\prime}$ (with an irreducible scheme $S^{\prime\prime}$)
and a section
$\delta^{\prime}: T^{\prime} \to S^{\prime\prime}$
of $\rho$ over $T^{\prime}$
such that the properties (1) and (2) from
Lemma
\ref{SmallAmountOfPoints} hold.
Set
$\delta^{\prime\prime}=\delta^{\prime} \circ \delta: T \to S^{\prime\prime}$
and
$\theta^{\prime\prime}_0=\theta_0 \circ \rho: S^{\prime\prime} \to U$.
We are also given with the
$S^{\prime\prime}$-group scheme isomorphisms
$$\Phi^{\prime\prime}_0=\rho^*(\Phi_0): (\theta^{\prime\prime}_0)^*(G_{\const,S}) \to (\theta^{\prime\prime}_0)^*(G_{\mathcal X}|_S) \ \text{and} \
\Psi^{\prime\prime}_0=\rho^*(\Psi_0): (\theta^{\prime\prime}_0)^*(C_{\const,S}) \to (\theta^{\prime\prime}_0)^*(C_{\mathcal X}|_S)$$
such that
$(\delta^{\prime\prime})^*(\Phi^{\prime\prime}_0)=\varphi$,
$(\delta^{\prime\prime})^*(\Psi^{\prime\prime}_0)=\psi$
and
$$(\theta^{\prime\prime}_0)^*(\mu_{\mathcal X}) \circ \Phi^{\prime\prime}_0 =
\Psi^{\prime\prime}_0 \circ (\theta^{\prime\prime}_0)^*(\mu_{\const}):
(\theta^{\prime\prime}_0)^*(G_{\const}) \to (\theta^{\prime\prime}_0)^*(C_{\mathcal X}).$$

We can extend these data to a
neighborhood $\mathcal V$ of $\{y_1,\dots,y_n\}$ and get the
diagram
\begin{equation}
\xymatrix{
     {}  &  S^{\prime\prime} \ar[d]^{\theta^{\prime\prime}_0} \ar@{^{(}->}@<-2pt>[r]  & {\mathcal V}^{\prime\prime}  \ar[d]_{\theta} &\\
     T \ar@{^{(}->}@<-2pt>[r] \ar[ur]^{
\delta^{\prime\prime}} & S \ar@{^{(}->}@<-2pt>[r]  &   \mathcal V \ar@{^{(}->}@<-2pt>[r]  &  \mathcal X &\\
    }
\end{equation}
\noindent
where
$\theta: {\mathcal V}^{\prime\prime}\to\mathcal V$
finite \'etale and
${\mathcal V}^{\prime\prime}$
is irreducible, and
isomorphisms
$$\Phi:\theta^*(G_{\const})\to\theta^*(G_{\mathcal X}) \ \text{and} \ \Psi:\theta^*(C_{\const})\to\theta^*(C_{\mathcal X})$$
such that
$$\theta^*(\mu_{\mathcal X}) \circ \Phi =
\Psi \circ \theta^*(\mu_{\const}):
\theta^*(G_{\const}) \to \theta^*(C_{\mathcal X}).$$

Since $T$ isomorphically projects onto $U$, it is still closed
viewed as a sub-scheme of $\mathcal V$. Note that since $\mathcal
Y$ is semi-local and $\mathcal V$ contains all of its closed
points, $\mathcal V$ contains $\Pi^{-1}(\Pi(\mathcal Y))=\mathcal
Y$. By Lemma \ref{Lemma_8_2} there exists an open subset $\mathcal
W\subseteq\mathcal V$ containing $\mathcal Y$ and endowed with a
finite surjective $U$-morphism $\Pi^*:\mathcal W\to\Aff^1\times
U$.
\par
Let $\mathcal X^{\prime}=\theta^{-1}(\mathcal W) \subseteq {\mathcal V}^{\prime\prime}$,
$f^{\prime}=\theta^{*}(f)$, $q^{\prime}_U=q_U\circ\theta$, and let
$\Delta^{\prime}:U\to\mathcal X^{\prime}$ be the section of
$q^{\prime}_U$ obtained as the composition of $\delta^{\prime\prime}$ with
$\Delta$. We claim that the triple
$(\mathcal X^{\prime},f^{\prime},\Delta^{\prime})$ is a nice triple over $U$. Let us
verify this. Firstly, the structure morphism
$q^{\prime}_U:\mathcal X^{\prime} \to U$ coincides with the
composition
$$
\mathcal X^{\prime}\xra{\theta}
\mathcal W\hra\mathcal X\xra{q_U} U.
$$
\noindent
Thus, it is smooth.
The scheme
$\mathcal X^{\prime}$
is irreducible as an open sub-scheme
of irreducible
${\mathcal V}^{\prime\prime}$.
The element $f^{\prime}$ belongs to the ring
$\Gamma(\mathcal X^{\prime},\mathcal O_{\mathcal X^{\prime}})$,
the morphism $\Delta^{\prime}$ is a section of $q^{\prime}_U$.
Each component of each fibre of the morphism $q_U$ has dimension
one, the morphism $\mathcal X^{\prime}\xra{\theta}\mathcal
W\hra\mathcal X$ is \'{e}tale. Thus, each component of each fibre
of the morphism $q^{\prime}_U$ is also of dimension one. Since
$\{f=0\} \subset {\mathcal W}$ and $\theta: \mathcal X^{\prime}
\to {\mathcal W}$ is finite,
the scheme
$\{f^{\prime}=0\}$ is finite over
$\{f=0\}$ and hence  also over $U$. In other words, the $\mathcal
O$-module $\Gamma(\mathcal X^{\prime},\mathcal O_{\mathcal
X^{\prime}})/f^{\prime}\cdot\Gamma(\mathcal X^{\prime},\mathcal
O_{\mathcal X^{\prime}})$ is finite. The morphism $\theta:
\mathcal X^{\prime}\to\mathcal W$ is finite and surjective.
We have constructed above in
Lemma \ref{Lemma_8_2}
the finite surjective morphism
$\Pi^*:\mathcal W\to\Aff^1\times U$. It follows that
$\Pi^*\circ\theta:\mathcal X^{\prime}\to\Aff^1\times U$ is finite
and surjective.
\par
Clearly, the \'{e}tale morphism
$\theta:\mathcal X^{\prime} \to\mathcal X$
is a morphism between the nice triples over $U$,
with $g^{\prime}=1$.
\par
Denote the restriction of $\Phi$ to $\mathcal X^{\prime}$ simply
by $\Phi$
and denote
the restriction of $\Psi$ to $\mathcal X^{\prime}$ simply
by $\Psi$. The equalities
$(\Delta^{\prime})^*{\Phi}=\id_{G_U}$
and
$(\Delta^{\prime})^*{\Psi}=\id_{C_U}$
hold by the very construction of the isomorphisms
$\Phi$ and $\Psi$.

All the closed points of the sub-scheme
$\{f=0\} \subset \mathcal X$
are in
$S$.
The morphism
$\theta$ is finite and
$\theta^{-1}(S)=S^{\prime\prime}$.
Thus all the closed points of the sub-scheme
$\{f^{\prime}=0\} \subset \mathcal X^{\prime}$
are in
$S^{\prime\prime}$.
Now the properties
(1) and (2)
from
Lemma
\ref{SmallAmountOfPoints}
of the
$U$-scheme
$S^{\prime\prime}$
show that
the assertions (2) and (3) of Theorem
\ref{ThmEquatingGroups}
do hold
for the closed sub-scheme
$\mathcal Z^{\prime}$ of $\mathcal X^{\prime}$ defined by $\{f^{\prime}=0\}$.
Theorem
\ref{ThmEquatingGroups}
follows.

\end{proof}

\section{One more Theorem}
\label{SectElemNisnevichSquare}
Theorem
\ref{ElementaryNisSquare} below
is a refinement of
\cite[Lemma 2]{OP}. It's proved in
\cite{Pan1}.
Let $k$ be a finite field and let
$U$ be as in Definition
\ref{DefnNiceTriple}.
For a $U$-scheme $V$ and any closed point $u \in U$
set $V_u=u\times_U V$.
\begin{thm}
\label{ElementaryNisSquare}
Let
$(q^{\prime}_U: \mathcal X^{\prime} \to U,f^{\prime},\Delta^{\prime})$
be a nice triple
over the scheme $U$ such that $f^{\prime}$ {\bf vanishes at every closed point of} $\Delta^{\prime}(U)$.
Let $\mathcal Z^{\prime}$ be the closed sub-scheme of $\mathcal X^{\prime}$ defined by $\{f^{\prime}=0\}$.
Assume that $\mathcal Z^{\prime}$ satisfies the conditions
{\rm (2)} and {\rm (3)} from
Theorem
\ref{ThmEquatingGroups}.
Then there exists a distinguished finite surjective morphism
$$ \sigma:\mathcal X^{\prime} \to\Aff^1\times U $$
\noindent
of $U$-schemes which enjoys the following properties:
\begin{itemize}
\item[\rm{(a)}]
the morphism
$\sigma|_{\mathcal Z^{\prime}}: \mathcal Z^{\prime} \to \Aff^1\times U$
is a closed embedding;
\item[\rm{(b)}] $\sigma$ is \'{e}tale in a neighborhood of
$\mathcal Z^{\prime}\cup\Delta^{\prime}(U)$;
\item[\rm{(c)}]
$\sigma^{-1}(\sigma(\mathcal Z^{\prime}))=\mathcal Z^{\prime}\coprod \mathcal Z^{\prime\prime}$
scheme theoretically and
$\mathcal Z^{\prime\prime} \cap \Delta^{\prime}(U)=\emptyset$;
\item[\rm{(d)}]
$\sigma^{-1}(\{0\} \times U)=\Delta^{\prime}(U)\coprod \mathcal D$ scheme theoretically and $\mathcal D \cap \mathcal Z^{\prime}=\emptyset$;
\item[\rm{(e)}]
for $\mathcal D_1:=\sigma^{-1}(\{1\} \times U)$ one has
$\mathcal D_1 \cap \mathcal Z^{\prime}=\emptyset$.
\item[\rm{(f)}]
there is a monic polinomial
$h \in \mathcal O[t]$
such that
$(h)=Ker[\mathcal O[t] \xrightarrow{\bar {} \circ \sigma^*}
\Gamma(\mathcal X^{\prime}, \mathcal O_{\mathcal X^{\prime}})/(f^{\prime})]$.
\end{itemize}
\end{thm}

\begin{cor}
\label{NormAtZeroAndOne}
Under the hypotheses of Theorem
\ref{ElementaryNisSquare}
let $K$ be the field of fractions of
$\mathcal O$,
$A=\Gamma(\mathcal X^{\prime}, \mathcal O_{\mathcal X^{\prime}})$,
$A_K=K \otimes_{\mathcal O} A$.
Consider an inclusion
$K[t] \subset A_K$ induced by the inclusion
$\sigma^*: \mathcal O[t] \hra A$.
Then
the polinomial $h$ from the item \emph{(f)}
does not vanish  at the points $t=1$ and $t=0$ of
the affine line
$\Aff^1_K$.

Let
$J^{\prime\prime} \subset A$
be the ideal defining the closed sub-scheme
$\mathcal Z^{\prime\prime}$
of the affine scheme
$\mathcal X^{\prime}$.
Then the extension $K[t] \subset A_K$,
the element
$f^{\prime} \in A_K$,
the polinomial
$h \in K[t]\cap f^{\prime}A_K$
from the item \emph{(\textrm{f})}
and the ideal
$J^{\prime\prime}_K$
satisfy the hypotheses of Lemma
\ref{KeyUnramifiedness}.
\end{cor}

\begin{proof}
The property $(e)$ of
\ref{ElementaryNisSquare}
implies that $h$ does not vanish at the point $1$.
Since
$(q^{\prime}_U: \mathcal X^{\prime} \to U,f^{\prime},\Delta^{\prime})$
is a nice triple
over $U$
one has
$(\Delta^{\prime})^*(f^{\prime}) \neq 0 \in \mathcal O$.
Since
$(\Delta^{\prime})^*(f^{\prime}) \neq 0 \in \mathcal O$
the conditions $(a)$ and $(d)$ imply that
$h$ does not vanish at $0$ either.

Prove now the second assertion. Firstly $A_K/fA_K$ is finite over $K$, since
$(q^{\prime}_U: \mathcal X^{\prime} \to U,f^{\prime},\Delta^{\prime})$
is a nice triple.
Secondly the items (a) and (f) show that
the induced map
$K[t]/hK[t]\to A_K/f^{\prime}A_K$ is an isomorphism.
Thirdly, the item (c)
shows that the ideal
$J^{\prime\prime}_K \subseteq A$
is co-prime to the ideal
$f^{\prime}A$.
Finally,
the items
(a), (f) and (c)
show that
$hA_K=f^{\prime}A_K\cdot J^{\prime\prime}$.

\end{proof}

\section{Proof of Theorem (A)}
\label{SectProofOfTheoremA}
\begin{proof}
Let $\mathcal O$ be the semi-local ring of finitely many closed points
$\{x_1,x_2,\dots,x_r\}$
on the $k$-smooth irreducible affine scheme $X$, where $k$ is {\bf the finite field}
from Theorem (A).
Set $U:= \text{Spec}(\mathcal O)$. Replacing $k$ with its algebraic closure
in $\Gamma(X, \mathcal O_X)$ we may and will assume that
$X$ is $k$-smooth and geometrically irreducible.
Further,
consider the reductive $U$-group scheme $G$, the $U$-tori $C$
and the smooth $U$-group scheme morphism
$$\mu : G \to C $$
from Theorem (A).
Let $K$ be the fraction field of $\mathcal O$.
Let
$\xi_K \in C(K)$
be such that
the element
$\bar \xi_K \in \mathcal F(K)$
is
$\mathcal O$-unramified (see (\ref{DefnUnramified})).

Shrinking $X$ if
necessary, we may secure the following properties:
\par\smallskip
(i) The points $x_1,x_2,\dots,x_r$ are still in $X$ and $X$ is affine;
\par\smallskip
(ii)
There are given a reductive $X$-group scheme $G_X$, an $X$-tori $C_X$, a smooth $X$-group scheme morphism
$\mu_X: G_X \to C_X$ such that $H_X:=ker(\mu_X)$ is a reductive $X$-group scheme and
$G=U \times_X G_X$, $C=U \times_X C_X$, $\mu=U \times_X \mu_X$. In particular,
for any $U$-scheme $S$ one has
$G(S)=G_X(S)$, $C(S)=C_X(S)$, $\mathcal F(S)=\mathcal F_X(S)$,
where
$\mathcal F_X(S):=C_X(S)/\mu_X(G_X(S))$.
\par\smallskip
(iii) The element $\xi$ is defined over $X_{\text{f}}$, that is
there is given an element
$\xi \in C_X(k[X]_{\text{f}})$
for a non-zero function
$\text{f} \in k[X]$
such that
the image of $\xi$ in $C_X(K)=C(K)$ coincides with the element $\xi_K$;
\par\smallskip
(iv) we may assume also that
$\bar \xi \in {\cal F}_X(k[X]_\textrm{f})$
is $k[X]$-unramified for the functor
$\mathcal F_X$;
\par\smallskip\noindent

Let
$\eta_{\text{f}}: \text{Spec}(K) \to X_{\text{f}}$
be a morphism induced by the inclusion
$k[X_{\text{f}}] \hra k(X)=K$
and let
$\eta: \text{Spec}(K) \to U$
be a morphism induced by the inclusion
$\mathcal O \hra K$.
Clearly,
$\xi_K= \eta^{*}_{\text{f}}(\xi) \in C_X(K)=C(K)$.

$\bullet$ \ {\it Our aim is to find an element
$\xi_U \in C_X(U)$
such that
$\eta^*(\bar \xi_U)=\bar \xi_K \in \mathcal F_X(K)$.}

%

We will construct such an element
$\xi_U$
rather explicitly in
(\ref{RightElement}).
At the moment right now we are given, in particular, with the smooth geometrically irreducible affine
$k$-scheme $X$, the finite family of points $x_1,x_2,\dots,x_n$ on
$X$, and the non-zero function $\text{f}\in k[X]$ vanishing at each
point $x_i$. Recall, that
beginning with these data
a nice triple
$$(q_U:\mathcal X\to U,f,\Delta)$$
is constructed in
\cite[Section 6, (16)]{PSV},
where
$\mathcal X=U \times_S X$
for a $k$-smooth affine scheme $S$ and a smooth morphism $X \to S$.
It is done shrinking $X$ and secure properties (i) to (iv) at the same time.
Recall that $q_U: \mathcal X=U \times_S X \to U$ is the projection to $U$. Let $q_X: \mathcal X= U \times_S X \to X$
be the projection to $X$.
\par
\begin{notation}
Set
$G_{\mathcal X}:=(q_X)^*(G_X)$
and
$G_{const}:=(q_U)^*(G)$,
$C_{\mathcal X}:=(q_X)^*(C_X)$
and
$C_{const}:=(q_U)^*(C)$.
Set
$\mu_{\mathcal X}=q^*_X(\mu_X): G_{\mathcal X} \to C_{\mathcal X}$
and
$\mu_{const}=q^*_U(\mu_{const}): G_{const} \to C_{const}$.
\end{notation}

By Theorem \ref{ThmEquatingGroups} there exist a morphism
\begin{equation}
\label{NiceTriplePrime}
\theta: (q^{\prime}_U: \mathcal X^{\prime} \to U, f^{\prime}, \Delta^{\prime}) \to
(q_U: \mathcal X \to U, f, \Delta)
\end{equation}
of nice triples over $U$ and isomorphisms
$$\Phi: \theta^*(G_{\text{const}}) \to \theta^*(G_{\mathcal X}), \ \Psi: \theta^*(C_{\text{const}}) \to \theta^*(C_{\mathcal X})$$
of
$\mathcal X^{\prime}$-group schemes such that
$(\Delta^{\prime})^*(\Phi)= id_{G_U}$,
$(\Delta^{\prime})^*(\Phi)= id_{G_U}$
and
\begin{equation}
\label{PhiAndPsiRelation}
\theta^*(\mu_{\mathcal X}) \circ \Phi = \Psi \circ \theta^*(\mu_{const})
\end{equation}
and the closed sub-scheme
$\mathcal Z^{\prime} \subset \mathcal X^{\prime}$
defined by $\{f^{\prime}=0\}$
satisfies the conditions
(2) and (3) from Theorem
\ref{ThmEquatingGroups}.

\begin{defn}
\label{F_const_F_X}
{\bf Set $q^{\prime}_X=q_X \circ \theta: \mathcal X^{\prime} \to X$.}
Define two functors $\mathcal F_{const}$ and \ $\mathcal F_{X}$ on the category of $\mathcal X^{\prime}$-schemes
as follows: for an $\mathcal X^{\prime}$-scheme $g: \mathcal Y^{\prime} \to \mathcal X^{\prime}$ set
$$\mathcal F_{const}(\mathcal Y^{\prime}):= C_{const}(\mathcal Y^{\prime})/\mu_{const}(G_{const}(\mathcal Y^{\prime})) \ \text{and} \ \
\mathcal F_X(\mathcal Y^{\prime}):=C_X(\mathcal Y^{\prime})/\mu_X(G_X(\mathcal Y^{\prime})).$$
Here for the functor $\mathcal F_{const}$ the scheme $\mathcal Y^{\prime}$ is regarded as a $U$-scheme via the composition
$q^{\prime}_U \circ g$ and for the functor $\mathcal F_X$ the scheme $\mathcal Y^{\prime}$ is regarded as an $X$-scheme
via the composition
$q^{\prime}_X \circ g$.
\end{defn}
The equality
(\ref{PhiAndPsiRelation})
implies that for every
$\mathcal X^{\prime}$-scheme
$\mathcal Y^{\prime}$
{\bf the group isomorphism}
$\Psi_{\mathcal Y^{\prime}}: \theta^*(C_{const})(\mathcal Y^{\prime}) \to \theta^*(C_{\mathcal X})(\mathcal Y^{\prime})$
induces a group isomorphism
\begin{equation}
\label{EquatingFunctors}
\bar \Psi_{\mathcal Y^{\prime}}:\mathcal F_{const}(\mathcal Y^{\prime})\to \mathcal F_X(\mathcal Y^{\prime})
\end{equation}
The group isomorphisms
$\Psi_{\mathcal Y^{\prime}}$
form a natural functor transformation of functors defined on the category of
$\mathcal X^{\prime}$-schemes and $\mathcal X^{\prime}$-scheme morphisms.
The same holds concerning isomorphisms
(\ref{EquatingFunctors}).

If an $\mathcal X^{\prime}$-scheme
is of the form
$\mathcal Y^{\prime} \to U \xra{\Delta^{\prime}} \mathcal X^{\prime}$
then
$\Psi_{\mathcal Y^{\prime}}=id$.
Indeed
$(\Delta^{\prime})^*(\Psi)= id_{C_U}$.
Particulary, we have proved the following
\begin{clm}
\label{PsiRavenId}
If $U$ is regarded as an $\mathcal X^{\prime}$-scheme via the morphism
$\Delta^{\prime}$
and $\textrm{Spec}(K)$
is regarded as an $\mathcal X^{\prime}$-scheme via the morphism
$\Delta^{\prime} \circ \eta$, then
$\Psi_U=id$ and $\Psi_K=id$.
\end{clm}
We will identify below in the proof the functors
$\mathcal F_{const}$
and
$\mathcal F_{X}$
via the functor isomorphism
$\bar \Psi$.
{\it We are rather closed to a construction of the required element $\xi_U \in C_X(U)$.}
Let
$(q^{\prime}_U: \mathcal X^{\prime}\to U, f^{\prime}, \Delta^{\prime})$
be the nice triple over $U$ from  (\ref{NiceTriplePrime}).
{\bf Since by \cite[Section 6, Claim 6.1]{PSV} $f$ vanishes at all closed points of $\Delta(U)$,
and $\theta$ is a morphism of nice triples, $f^{\prime}$ vanishes at all closed points of $\Delta^{\prime}(U)$ as well}.
As was already mentioned in this proof the closed sub-scheme
$\mathcal Z^{\prime} \subset \mathcal X^{\prime}$
defined by $\{f^{\prime}=0\}$
satisfies the conditions
(2) and (3) from Theorem
\ref{ThmEquatingGroups}.
So, the nice triple
$(q^{\prime}_U: \mathcal X^{\prime}\to U, f^{\prime}, \Delta^{\prime})$
satisfies the hypotheses of Theorem
\ref{ElementaryNisSquare}.

By Theorem
\ref{ElementaryNisSquare}
there exists
a finite surjective morphism
$$ \sigma:\mathcal X^{\prime} \to\Aff^1\times U $$
\noindent
of $U$-schemes which enjoys the properties:
{\bf subjecting conditions (a) to (e)
from that Theorem. }
Since ${\cal X^{\prime}}$ and
${\cal Y^{\prime}}:=U \times \Aff^1$
are regular schemes and $\sigma$ is finite surjective
it is finite and flat by a theorem of Grothendieck
\cite[Cor.18.17]{E}.
So, for any $U$-map
$t:{\cal Y^{\prime\prime}}\to {\cal Y^{\prime}}$
of affine $U$-schemes, setting
${\cal X^{\prime\prime}}={\cal X^{\prime}}\times_{\cal Y^{\prime}} {\cal Y^{\prime\prime}}$
we have the norm map
given by formulae
(\ref{NormMap})
$$
N_{\mathcal X^{\prime\prime}/\mathcal Y^{\prime\prime}}: C_{const}({\cal X^{\prime\prime}}) \to C_{const}(\cal Y^{\prime\prime}).
$$
Let
${\cal D}_1=\sigma^{-1}(U\times \{1\})$ and ${\cal D}$ be as in Theorem
\ref{ElementaryNisSquare}.
The scheme theoretic pre-image
$\sigma^{-1}(U\times \{0\})$
of
$U \times \{0\}$
is equal to
$\Delta^{\prime}(U)\coprod \mathcal D$.

Consider the element
$\xi \in C_X(k[X]_\textrm{f})$
chosen above.
Set
$$\zeta= (q_X^{\prime})^*(\xi) \in C_X(\mathcal X^{\prime}_{f^{\prime}})=C_{const}(\mathcal X^{\prime}_{f^{\prime}}).$$
By Theorem
\ref{ElementaryNisSquare}
one has
$\mathcal D_1 \cap \mathcal Z^{\prime}=\emptyset$
and
$\mathcal D \cap \mathcal Z^{\prime}=\emptyset$.
Thus one can form the following element
\begin{equation}
\label{RightElement}
\xi_U = N_{{\cal D}_1/U}(\zeta|_{{\cal D}_1})N_{{\cal D}/U}(\zeta|_{{\cal D}})^{-1}
\in C_{const}(U)=C_X(U)
\end{equation}

\begin{clm}[Main]
\label{MainClaim}
One has
$\bar \xi_{K}= \eta^*(\bar \xi_U) \in \mathcal F_X(K)$,
where $K$ is the fraction field  of both $k[X]$ and
$\mathcal O$.
\end{clm}
To complete the proof of Theorem A, it remains to prove the Claim.
To do this it is convenient to fix some notations.
For
an $U$-scheme ${\cal Z}$
set
${\cal Z}_K= \textrm{Spec}(K) \times_U {\cal Z}$,
for a
$U$-morphism
$\varphi: {\cal Z} \to {\cal W}$
set
$\varphi_K= \textrm{Spec}(K) \times_U \varphi$.
Clearly one has
$U_K= \textrm{Spec}(K)$
and
$(U \times \Aff^1)_K$
is just the affine line
$\Aff^1_K$.
Its closed subschemes
$U_K \times \{1\}$
and
$U_K \times \{0\}$
coincide with the points
$1$ and $0$ of $\Aff^1_K$.
The morphism
$(q^{\prime}_U)_K: {\cal X^{\prime}}_K \to U_K=\textrm{Spec}(K)$
is smooth. In fact, the morphism $\theta$ is \'{e}tale and
$q_U: \mathcal X \to U$ is smooth.
The morphism $\sigma_K$ is finite flat and fits in the commutative triangle
\begin{equation}
    \xymatrix{
    {\cal X^{\prime}}_K \ar[rr]^-{\sigma_K} \ar[rd]_-{(q^{\prime}_U)_K} && \Aff^1_K \ar[ld]^-{pr_K} &   \\
    & \textrm{Spec}(K)  \\
    }
\end{equation}
Clearly
${\cal D}_{1,K}$ is the scheme theoretic pre-image of
the point $\{1\} \in \Aff^1_K$
and the disjoint union
${\cal D}_K \coprod \Delta^{\prime}_K$
is the scheme theoretic pre-image of the point
$\{0\} \in \Aff^1_K$
under the morphism
$\sigma_K$.

Let
$\zeta_{K}$
be the pull-back of
$\zeta$
under the natural map
$\mathcal X^{\prime}_{f^{\prime},K} \to \mathcal X^{\prime}_{f^{\prime}}$.
Let
$\eta: \textrm{Spec}(K)= U_K \to U$
be the map from the beginning of the present section.
By the base change property of the norm map
(\ref{NormMap}) one has
\begin{equation}
\label{XiPrimeUK}
\eta^*(\xi_U)=\xi_{U,K} = N_{{\cal D}_{1,K}/U_K}(\zeta_{K}|_{{\cal D}_{1,K}})N_{{\cal D_K}/U_K}(\zeta_{K}|_{{\cal D}_K})^{-1}
\in C_{const}(K)
\end{equation}
Set
\begin{equation}
\label{ZetaU}
\zeta_t:=N_{K({\cal X^{\prime}}_K)/K(\Aff^1)}(\zeta_{K})  \in C_{const}(K(\Aff^1))=C_{const}(K(t)).
\end{equation}
To prove Claim
\ref{MainClaim}
we need the following one:
\begin{clm}
\label{Claim}
The class
$\bar \zeta_t \in {\cal F}_{const}(K(t))$
is $K[t]$-unramified for the functor
$\mathcal F_{const}$.
\end{clm}
{\it Prove now this Claim.}
Let
$\zeta_{K} \in C_X(\mathcal X^{\prime}_{f^{\prime},K})$
be the pull-back of
$\zeta \in C_X(\mathcal X^{\prime}_{f^{\prime}})$
under the natural morphism
$\mathcal X^{\prime}_{f^{\prime},K} \to \mathcal X^{\prime}_{f^{\prime}}$.

Consider the composition morphism of schemes
$r: \mathcal X^{\prime}_{K} \to \mathcal X^{\prime} \xra{q^{\prime}_X} X$.
It induces a rational function field inclusion
$K \hra K(\mathcal X^{\prime}_{K})$
(stress that it does not coincides with the one induced by the projection
$q^{\prime}_{U,K}: \mathcal X^{\prime}_{K} \to \textrm{Spec}(K)$).
It's easy to check that for each closed point
$y \in \mathcal X^{\prime}_{K}$
its image
$r(y) \in X$
has either hight $1$ or $0$ (codimension $1$ or $0$).
The family of commutative diagrams
\begin{equation}
\label{UnramifienessDiagram}
    \xymatrix{
         &&  \mathcal F_X(\mathcal O_{\mathcal X^{\prime}_{K},y}) \ar[d]^{}  && \mathcal F_X(\mathcal O_{X,r(y)}) \ar[ll]_{r^*} \ar[d]_{} &\\
     && \mathcal F_X(K(\mathcal X^{\prime}_{K})) &&  \ar[ll]_{r^*} \mathcal F_X(K).  & \\
    }
\end{equation}
shows that the map
$r^*: \mathcal F_X(K) \to \mathcal F_X(K(\mathcal X^{\prime}_{K}))$
takes
$X$-unramified elements to
$\mathcal X^{\prime}_{K}$-unramified elements for the functor $\mathcal F_X$.
Since the class
$\bar \xi \in {\cal F}_X(k[X]_\textrm{f})$
is $k[X]$-unramified,
the element
$\bar \zeta_{K} =r^*(\bar \xi) \in \mathcal F_X(K(\mathcal X^{\prime}_{f^{\prime},K}))=\mathcal F_{const}(\mathcal X^{\prime}_{f^{\prime},K})$
is
$\mathcal X^{\prime}_{K}$-unramified.

{\bf Under the notation of Corollary
\ref{NormAtZeroAndOne}
the extension $K[t] \subset A_K$,
the element
$f^{\prime} \in A_K$,
the polinomial
$h \in K[t]\cap f^{\prime}A_K$
and the ideal
$J^{\prime\prime}_K$
satisfy the hypotheses of Lemma
\ref{KeyUnramifiedness}.
As we already know the element
$\bar \zeta_{K} \in \mathcal F_{const}(\mathcal X^{\prime}_{f^{\prime},K})$
is
$\mathcal X^{\prime}_{K}$-unramified
for the functor $\mathcal F_{const}$.

Thus by Lemma
\ref{KeyUnramifiedness}
the class $\bar \zeta_t$
is $K[t]$-unramified for the functor
$\mathcal F_{const}$.
This implies the Claim
\ref{Claim}.}

Continue the proof of Claim
\ref{MainClaim}.
By Claim
\ref{Claim}
we can apply the specialization maps to
$\bar \zeta_t$.
By Corollary
\ref{TwoSpecializations}
the specializations at $0$ and $1$ of the element
$\bar \zeta_t$
coincide, that is
\begin{equation}
\label{Rel2}
s_1(\bar \zeta_t)=s_0(\bar \zeta_t) \in {\cal F}_{const}(K).
\end{equation}
By Corollary
\ref{NormAtZeroAndOne}
the function
$h \in K[t]$
does not vanish as at $1$, so at $0$ and
by the items (a) and (f) of Theorem
\ref{ElementaryNisSquare}
one has
$\zeta_t \in C_{const}(K[t]_{h})$.
Using the relation between specialization and evaluation maps described in
Definition
\ref{SpecializationDef}
one has a chain of equalities
\begin{equation}
\label{Rel3}
\overline {Ev_1(\zeta_t)} = s_1(\bar \zeta_t)= s_0(\bar \zeta_t)=
\overline {Ev_0(\zeta_t)} \in {\cal F}_{const}(K).
\end{equation}
The base change and the multiplicativity properties of the norm map
(\ref{NormMap})
imply equalities
\begin{equation}
\label{Rel4}
Ev_1(\zeta_t)= N_{{\cal D}_{1,K}/U_K}(\zeta_{K}|_{{\cal D}_{1,K}})
\end{equation}
and
\begin{equation}
\label{Rel4}
Ev_0(\zeta_t)=
N_{{\cal D}_K \coprod \Delta^{\prime}_K/U_K}(\zeta_{K}|_{{\cal D}_K \coprod \Delta^{\prime}_K})=
N_{{\cal D}_K/U_K}(\zeta_{K}|_{{\cal D}_K})\cdot N_{\Delta^{\prime}_K/U_K}(\zeta_{K}|_{\Delta^{\prime}_K})
\end{equation}
So, we have a chain of equalities in ${\cal F}_{const}(K)$
\begin{equation}
\label{Rel4a}
\overline {N_{{\cal D}_{1,K}/U_K}(\zeta_{K}|_{{\cal D}_{1,K}})}=
\overline {Ev_1(\zeta_t)}=\overline {Ev_0(\zeta_t)}=
\overline {N_{{\cal D}_K/U_K}(\zeta_{K}|_{{\cal D}_K})}\cdot \overline {N_{\Delta^{\prime}_K/U_K}(\zeta_{K}|_{\Delta^{\prime}_K})}
\end{equation}
By the normalization property of the norm map
one has
\begin{equation}
\label{Rel4b}
N_{\Delta^{\prime}_K/U_K}(\zeta_{K}|_{\Delta^{\prime}_K})=(\Delta^{\prime}_K)^*(\zeta_{K})
\end{equation}

\begin{clm}
\label{Claim2}
$(\Delta^{\prime}_K)^*(\zeta_{K})=\xi_{K} \in C_{const}(K)$.
\end{clm}
Set
$q^{\prime}_{X,\textrm{f}}= q^{\prime}_{X}|_{\mathcal X^{\prime}_{f^{\prime}}}:
\mathcal X^{\prime}_{f^{\prime}} \to X_{\textrm{f}}$,
where $q^{\prime}_{X}$ is from Definition
\ref{F_const_F_X}.
Let $r_{\textrm{f}}$ be the composite morphism
$\mathcal X^{\prime}_{f^{\prime},K} \to \mathcal X^{\prime}_{f^{\prime}} \xra{q^{\prime}_{X, \textrm{f}}} X_{\textrm{f}}$.
Clearly,
$r_{\textrm{f}} \circ \Delta^{\prime}_K= \eta_{\textrm{f}}: Spec(K) \to X_{\textrm{f}}$.
The following chain of equalities prove Claim
\ref{Claim2}:
$(\Delta^{\prime}_K)^*(\zeta_{K})=(\Delta^{\prime}_K)^*(r^*_{\textrm{f}}(\xi))=
\eta^*_{\textrm{f}}(\xi)=\xi_K \in C_X(K)=C_{const}(K)$. \\
By Claim \ref{Claim2} and equalities
(\ref{Rel4b}), (\ref{Rel4a}), (\ref{XiPrimeUK})
we have
in
${\cal F}_{const}(K)={\cal F}_X(K)$
$$
\bar \xi_{K}= (\Delta^{\prime}_K)^*(\bar \zeta_{K})=
\overline {(\Delta^{\prime}_K)^*(\zeta_{K})}=
\overline {N_{\Delta^{\prime}_K/U_K}(\zeta_{K}|_{\Delta^{\prime}_K})}=
$$
$$
\overline {N_{{\cal D}_{1,K}/U_K}(\zeta_{K}|_{{\cal D}_{1,K}})} \cdot
(\overline {N_{{\cal D}_K/U_K}(\zeta_{K}|_{{\cal D}_K})})^{-1}=\eta^*(\bar \xi_U).
$$

Whence the Claim
\ref{MainClaim}.
The proof of {\bf Theorem A} is completed.

\end{proof}

\section{An extension of Theorem A}
\label{SectProofofTheoremB}
The main aim of the present Section is to prove the following result
\begin{thm}[Theorem B]
\label{PurityGeneral}
Let $R$ be a regular local
domain containing {\bf a field $k$.}
Let
$$\mu: G\to C$$
be a smooth $R$-group scheme morphism of reductive
$R$-group schemes, with a torus $C$. Set
$H= ker (\mu)$ and suppose additionally that
$H$ is a reductive $\mathcal O$-group scheme.
The functor
$$\mathcal F: S\mapsto C(S)/\mu(G(S))$$
defined on the category of $R$-algebras
satisfies purity for $R$.
\end{thm}

\begin{proof}
To prove Theorem B
we now recall a celebrated result of Dorin
Popescu (see \cite{P}
or, for a self-contained proof, \cite{Sw}).

Let $k$ be a field and $R$ a local $k$-algebra. We say that $R$ is
{\it geometrically regular} if $k'\otimes_kR$ is regular for any
finite extension $k'$ of $k$. A ring homomorphism $A\to R$ is
called {\it geometrically regular} if it is flat and for each
prime ideal $\mathfrak q$ of $R$ lying over $\mathfrak p$, $R_\mathfrak q/\mathfrak p R_\mathfrak q =
k(\mathfrak p)\otimes_AR_\mathfrak q$ is geometrically regular over
$k(\mathfrak p)=A_\mathfrak p/\mathfrak p_\mathfrak p$.

Observe that any regular local ring containing a {\bf perfect} field $k$ is
geometrically regular over $k$.

\begin{thm}[Popescu's theorem] A  homomorphism $A\to R$  of
noetherian rings is geometrically regular if and only if $R$ is a
filtered direct limit of smooth $A$-algebras.
\end{thm}

{\it Proof of Theorem B.}
Let $R$ be a regular local ring
containing {\bf a field $k$.}
If $char(k)$ is zero then $k$ is perfect.
If $char(k)=p>0$, then we replace $k$ with the field $\mathbb F_p$.
In both cases $k$ is perfect.
Since $k$ is perfect
one can apply Popescu's theorem. So, $R$ can be presented as a filtered
direct limit of smooth $k$-algebras $A_{\alpha}$ over the field $k$.
We
first observe that we may replace the direct system of the
$A_\alpha$'s by a system of  essentially smooth local
$k$-algebras. In fact, if
${\mathfrak  m}$
are all maximal ideal of $R$ and
$S=R- {\mathfrak  m}$
we can replace each
$A_\alpha$
by
$\left(A_\alpha\right)_{S_{\alpha}}$,
where
$S_\alpha=S \cap A_\alpha$. Note
that in this case the canonical morphisms
$\varphi_\alpha:A_\alpha\to R$
are local (sends the maximal ideal to the maximal one) and every
$A_\alpha$
is a regular local ring, in particular a factorial ring.

Let now $L$ be the field of fractions of $R$ and, for each
$\alpha$,  let $K_\alpha$ be the field of fractions of $A_\alpha$.
For each index $\alpha$ let
$\mathfrak a_{\alpha}$
be the kernel of the map
$\varphi_{\alpha}: A_{\alpha} \to R$
and
$B_{\alpha}=(A_{\alpha})_{\mathfrak a_{\alpha}}$.
Clearly, for each
$\alpha$,  $K_{\alpha}$
is the field of fractions of
$B_{\alpha}$.
The composition map
$A_{\alpha} \to R \to L$
factors through $B_{\alpha}$
and hence it also factors through the residue field $k_{\alpha}$
of $B_{\alpha}$.
Since $R$ is a filtering direct limit of the $A_{\alpha}$'s
we see that $L$ is a filtering direct limit of the $B_{\alpha}$'s.
We will write
$\psi_{\alpha}$
for the canonical morphism
$B_{\alpha} \to L$.

Let $\xi \in C(L)$ be such that the class
$\bar \xi \in {\cal F}(L)$
is $R$-unramified. We need the following two lemmas.
\begin{lem}
\label{WeakGrothendieck}
Let $B$ be a regular local ring
and let
$K$ be its field of fractions.
Let ${\mathfrak m}$ be a maximal ideal of $B$
and
$\bar B= B/{\mathfrak m}$.
For an element
$\theta \in {\cal F}(B)$
write
$\bar \theta$ for its image in
${\cal F}(\bar B)$
and
$\theta_K$
for its image in
${\cal F}(K)$.
Let
$\eta, \rho \in {\cal F}(B)$
be such that
$\eta_K = \rho_K \in {\cal F}(K)$.
Then
$\bar \eta = \bar \rho \in {\cal F}(\bar B)$.
\end{lem} 

\begin{lem}
\label{LiftToFiniteLevel}
There exists an index $\alpha$ and an element
$\xi_{\alpha} \in C(B_{\alpha})$
such that
$\psi_{\alpha}(\xi_{\alpha})=\xi$
and the class
$\bar \xi_{\alpha} \in {\cal F}(K_{\alpha})$
is $A_{\alpha}$-unramified.
\end{lem}
Assuming these two Lemmas we complete the proof as follows.
Consider a commutative diagram
$$
\xymatrix{
    A_{\alpha}  \ar[d]\ar[rr]^-{\varphi_{\alpha}}\ar[d]_-{} && R \ar[d]^-{} &   \\
     B_{\alpha} \ar[d]\ar[r]^{}  & k_{\alpha} \ar[r]^{} & L \\
     K_{\alpha}.
    }
$$
By Lemma
\ref{LiftToFiniteLevel}
the class
$\bar \xi_{\alpha} \in {\cal F}(K_{\alpha})$
is $A_{\alpha}$-unramified.
Hence by Theorem $A$ there exists an element
$\eta \in C(A_{\alpha})$
such that
$\bar \xi_{\alpha}=\bar \eta \in {\cal F}(K_{\alpha})$.
By Lemma
\ref{WeakGrothendieck}
the elements
$\bar \xi_{\alpha}$
and
$\bar \eta$
have the same image in ${\cal F}(k_{\alpha})$.
Hence $\bar \xi \in {\cal F}(L)$ coincides with
the image of the element $\varphi_{\alpha} (\bar \eta)$
in ${\cal F}(L)$.
It remains to prove the two Lemmas.

{\it Proof of Lemma \ref{WeakGrothendieck}.}
Induction on $dim (B)$. The case of dimension $1$ follows from Theorem
\ref{NisnevichCor} applied to the local ring
$B$. To prove the general case choose an $f \in B$
such that $\eta = \rho \in {\cal F}(B_f)$. Let $\pi \in B$
be such that
$\pi$ is a regular parameter in
$B$, having no common factors with $f$. Let
$B^{\prime}= B/\pi B$. Then for the image
$(\eta - \rho)^{\prime}$ of $\eta - \rho$
in
${\cal F}(B^{\prime})$
we have
$(\eta - \rho)^{\prime}_f =0
\in {\cal F}(B^{\prime}_{f})$.
By the inductive hypotheses one has
$\overline {(\eta - \rho)^{\prime}}=0 \in {\cal F}(\bar B)$.
Since
$\overline {(\eta - \rho)}= \overline {(\eta - \rho)^{\prime}} \in {\cal F}(\bar B)$,
one has
$\bar \eta = \bar \rho \in {\cal F}(\bar B)$.

{\it Proof of Lemma \ref{LiftToFiniteLevel}.}

Choose an $f \in R$ such that $\xi$ is defined over $R_f$.
Then $\xi$ is ramified at most at those hight one primes
$\mathfrak p_1, \dots, \mathfrak p_r$
which contains $f$. Since the class
$\bar \xi \in {\cal F}(L)$
is $R$-unramified there exists, for any
$\mathfrak p_i$, an element
$\sigma_i \in G(L)$
and an element
$\xi_i \in C(R_{\mathfrak p_i})$
such that
$\xi=\mu(\sigma_i) \xi_i \in C(L)$.
We may assume that $\xi_i$ is defined over $R_{h_i}$
for some
$h_i \in R - \mathfrak p_i$
and that $\sigma_i$ is defined over $R_{g_i}$
for some $g_i \in R$.

We can find an index $\alpha$ such that $A_{\alpha}$
contains lifts
$f_{\alpha}, h_{1,\alpha}, \dots, h_{r,\alpha}, g_{1\alpha}, \dots, g_{r,\alpha}$
and moreover
\begin{itemize}
\item[(1)]
$C(A_{\alpha, f_{\alpha}})$ contains a lift $\xi_{\alpha}$ of $\xi$,
\item[(2)]
$C(A_{\alpha, h_{i,\alpha}})$ contains a lift of $\xi_{i,\alpha}$ of $\xi_i$,
\item[(3)]
$G(A_{\alpha, g_{i,\alpha}})$ contains a lift of $\sigma_{i,\alpha}$ of $\sigma_i$.
\end{itemize}
Since none of the
$f_{\alpha}, h_{1,\alpha}, \dots, h_{r,\alpha}, g_{1\alpha}, \dots, g_{r,\alpha}$
 vanishes in $R$, the elements
$\xi_{\alpha}, \xi_{1,\alpha}, \dots, \xi_{ir,\alpha}$
and
$\sigma_{1,\alpha}, \dots, \sigma_{r,\alpha}$
may be regarded as elements of
$C(B_{\alpha})$
and
$G(B_{\alpha})$
respectively.

We know that
$\xi_{i,\alpha} \mu (\sigma_{i,\alpha})$
and
$\xi_{\alpha}$
map to the same element in $C(L)$.
Hence replacing $\alpha$ by a larger index, we may assume that
$\xi_{\alpha}= \xi_{i,\alpha} \mu (\sigma_{i,\alpha}) \in C(B_{\alpha})$.
We claim that the class
$\bar \xi_{\alpha} \in {\cal F}(K_{\alpha})$
is $A_{\alpha}$-unramified.
To prove this note that
the only primes at which $\bar \xi_{\alpha}$ could be ramified are those which divide
$f_{\alpha}$. Let
$\mathfrak q_{\alpha}$
be one of them. Check that $\bar \xi_{\alpha}$ is
unramified at $\mathfrak q_{\alpha}$. Let $q_{\alpha} \in A_{\alpha}$ be
a prime element such that
$q_{\alpha}A_{\alpha}=\mathfrak q_{\alpha}$. Then
$q_{\alpha}r_{\alpha} = f_{\alpha}$
for an element
$r_{\alpha}$.
Thus $qr=f \in R$ for the images of $q_{\alpha}$ and $r_{\alpha}$ in $R$.
Since the homomorphism $\varphi_{\alpha}: A_{\alpha} \to R$ is local,
$q \in \mathfrak m_R$. The relation $qr=f$ shows that
$q \in \mathfrak p_i$ for some index $i$. Thus
$q_{\alpha} \in \varphi_{\alpha}^{-1}(\mathfrak p_i)$
and
$\mathfrak q_{\alpha} \subset \varphi_{\alpha}^{-1}(\mathfrak p_i)$.
On the other hand
$h_{i, \alpha} \in A_{\alpha} - \varphi_{\alpha}^{-1}(\mathfrak p_i)$,
because
$h_i \in R - \mathfrak p_i$. Thus
$h_{i, \alpha} \in A_{\alpha} - \mathfrak q_{\alpha}$.
Now the relation
$\xi_{\alpha}= \xi_{i,\alpha} \mu (\sigma_{i,\alpha}) \in C(B_{\alpha})$
with
$\xi_{i,\alpha} \in C(A_{\alpha, h_{i,\alpha}})$
shows that
$\bar \xi_{\alpha}$ is unramified at
$\mathfrak q_{\alpha}$.
Thus
$\bar \xi_{\alpha}$
is unramified at each hight one prime in $A_{\alpha}$ containing $f_{\alpha}$.
Since
$\xi_{\alpha} \in C(A_{\alpha, f_{\alpha}})$
we conclude that
$\bar \xi_{\alpha}$
is $A_{\alpha}$-unramified.
The lemma follows. The theorem is proved.

\end{proof}

\section{One more purity result}
\label{SectionOneMorePurityTheorem}
In this Section we prove another purity theorem for reductive group schemes.
Let $k$ be {\bf a finite field}. Let $\mathcal O$ be the semi-local ring of finitely many {\bf closed} points
on a $k$-smooth {\bf irreducible affine} $k$-variety $X$ and let $K$ be the field of fractions of $\mathcal O$.
Let $G$ be a semi-simple $\mathcal O$-group scheme.
Let
$i: Z \hra G$ be a closed subgroup scheme of the center $Cent(G)$.
{\bf It is known that $Z$ is of multiplicative type}.
Let $G'=G/Z$ be the factor group,
$\pi: G \to G'$ be the projection.
It is known that $\pi$ is finite surjective and strictly flat. Thus
the sequence of $\mathcal O$-group schemes
\begin{equation}
\label{ZandGndGprime}
\{1\} \to Z \xra{i} G \xra{\pi} G^{\prime} \to \{1\}
\end{equation}
induces an exact sequence of group sheaves in $\text{fppt}$-topology.
Thus for every $\mathcal O$-algebra $R$ the sequence
(\ref{ZandGndGprime})
gives rise to a boundary operator
\begin{equation}
\label{boundary}
\delta_{\pi,R}: G'(R) \to \textrm{H}^1_{\text{fppt}}(R,Z)
\end{equation}
One can check that it is a group homomorphism
(compare \cite[Ch.II, \S 5.6, Cor.2]{Se}).
Set
\begin{equation}
\label{AnotherFunctor}
{\cal F}(R)= \textrm{H}^1_{\text{fttp}}(R,Z)/ Im(\delta_{\pi,R}).
\end{equation}
Clearly we get a functor on the category of $\mathcal O$-algebras.
\begin{thm}
\label{PurityForSubgroup}
Let $\mathcal O$ be the semi-local ring of finitely many {\bf closed} points
on a $k$-smooth irreducible {\bf affine} $k$-variety $X$.
Let $G$ be a semi-simple $\mathcal O$-group scheme.
Let
$i: Z \hra G$ be a closed subgroup scheme of the center $Cent(G)$.
Then
the functor
${\cal F}$
on the category
$\mathcal O$-algebras
given by
(\ref{AnotherFunctor})
satisfies purity for the ring $\mathcal O$
regarded as an
$\mathcal O$-algebra
via the identity map.

If $K$ is the fraction field of $\mathcal O$ this statement can be restated in an explicit way
as follows:
given an element
$\xi \in \textrm{H}^1_{\text{fppt}}(K,Z)$
suppose that for each height $1$ prime ideal
$\mathfrak p$ in $\mathcal O$ there exists
$\xi_{\mathfrak p } \in \textrm{H}^1_{\text{fppt}}(\mathcal O_{\mathfrak p }, Z)$,
$g_{\mathfrak p } \in G'(K)$
with
$\xi=\xi_{\mathfrak p } + \delta_{\pi}(g_{\mathfrak p }) \in \textrm{H}^1_{\text{fppt}}(K,Z)$.
Then there exists
$\xi_{\mathfrak m } \in \textrm{H}^1_{\text{fppt}}(\mathcal O, Z)$, $g_{\mathfrak m } \in G'(K)$,
such that
$$
\xi=\xi_{\mathfrak m } + \delta_{\pi}(g_{\mathfrak m }) \in \textrm{H}^1_{\text{fppt}}(K,Z).
$$

\end{thm}

\begin{proof}
The group $Z$ is of  multiplicative type.
So we can find a finite \'{e}tale $\mathcal O$-algebra $A$ and
a closed embedding
$Z \hra R_{A/{\mathcal O}}(\mathbb G_{m,\ A})$
into the permutation torus
$T^{+}=R_{A/\mathcal O}(\mathbb G_{m,\ A})$.
Let
$G^{+}=(G \times T^{+})/Z$
and
$T=T^{+}/Z$,
where
$Z$ is embedded in
$G \times T^{+}$
diagonally.
Clearly
$G^{+}/G=T$.
Consider a commutative diagram
$$
\xymatrix {
{}           & \{1\}                       & \{1\} \\
{}           & {G'} \ar[r]^{id}  \ar[u]          & {G'}   \ar[u]          \\
\{1\} \ar[r] & {G} \ar[r]^{j^+} \ar[u]^{\pi} & {G^+} \ar[r]^{\mu^+} \ar[u]^{\pi^+} & {T} \ar[r] & \{1\} \\
\{1\} \ar[r] & {Z} \ar[r]^{j} \ar[u]^{i} & {T^+} \ar[r]^{\mu} \ar[u]^{i^+} & {T} \ar[u]_{id} \ar[r] &  \{1\} \\
{}           & \{1\} \ar[u]                      & \{1\} \ar[u]\\
}
$$
with exact rows and columns.
By
\ref{FlatAndEtale}
and Hilbert 90
for the semi-local $\mathcal O$-algebra $A$ one has
$\textrm{H}^1_{\text{fppt}}(\mathcal O,T^+)= \textrm{H}^1_{\text{\'et}}(\mathcal O,T^+)= \textrm{H}^1_{\text{\'et}}(A,\Bbb G_{m,A})=\{* \}$.
So, the latter diagram gives rise to
a commutative diagram of pointed sets
$$
\xymatrix {
{}           & {}                            & {\textrm{H}^1_{\text{fppt}}(\mathcal O,G')} \ar[r]^{id}            & {\textrm{H}^1_{\text{fppt}}(\mathcal O,G')}            \\
{G^+(\mathcal O)} \ar[r]^{\mu^+_\mathcal O} & {T(\mathcal O)} \ar[r]^{\delta^+_\mathcal O}  & {\textrm{H}^1_{\text{fppt}}(\mathcal O,G)} \ar[r]^{j^+_*} \ar[u]^{\pi_*} & {\textrm{H}^1_{\text{fppt}}(\mathcal O,G^+)} \ar[u]^{\pi^+_*}  \\
{T^+(\mathcal O)} \ar[r]^{\mu_\mathcal O} \ar[u]^{i^+_*} & {T(\mathcal O)} \ar[r]^{\delta_\mathcal O} \ar[u]^{id} & {\textrm{H}^1_{\text{fppt}}(\mathcal O,Z)} \ar[r]^{\mu} \ar[u]^{i_*} & {\{*\}} \ar[u]^{i^+_*}  \\
{}           & {} & {G'(\mathcal O)} \ar[u]^{\delta_{\pi}}                      & {} \\
}
$$
with exact rows and columns. It follows that
$\pi^+_*$ has  trivial kernel and one has a chain of group isomorphisms
$$
\textrm{H}^1_{\text{fppt}}(\mathcal O,Z) / Im (\delta_{\pi,\mathcal O})= ker (\pi_*)= ker (j^+_*) = T(\mathcal O)/\mu^+ (G^+(\mathcal O)).
$$
Clearly these isomorphisms respect $\mathcal O$-homomorphisms of semi-local $\mathcal O$-algebras.

The morphism $\mu^{+}: G^{+} \to T$ is a smooth $\mathcal O$-morphism of reductive
$\mathcal O$-group schemes, with the torus $T$. The kernel
$ker(\mu^{+})$
is equal to $G$ and $G$ is a reductive $\mathcal O$-group scheme.
The functor
$\mathcal O^{\prime} \mapsto T(\mathcal O^{\prime})/\mu^+ (G^+(\mathcal O^{\prime}))$
satisfies purity for the regular semi-local $\mathcal O$-algebra $\mathcal O$ by Theorem (A).
Hence the functor
$\mathcal O^{\prime} \mapsto \textrm{H}^1_{\text{fppt}}(\mathcal O^{\prime},Z) / Im (\delta_{\pi,\mathcal O^{\prime}})$
satisfies purity for $\mathcal O$.

\end{proof}

\section{Proof of Theorem \ref{MainThmGeometric}}
\label{SectionGr_SerreConj}

\begin{proof}[Proof of Theorem \ref{MainThmGeometric} for the case of semi-simple reductive group scheme]
Let $\mathcal O$ and $G$ be the same as in Theorem \ref{MainThmGeometric} and assume additionally that $G$ is semi-simple.
We need to prove that
\begin{equation}
\label{Gr_SerreForGsemisimple}
ker[H^1_{\text{\'et}}(\mathcal O,G) \to H^1_{\text{\'et}}(K,G)]=* .
\end{equation}

Let
$G^{sc}_{}$
be the corresponding simply-connected semi-simple $\mathcal O$-group scheme
and let
$\pi: G^{sc}_{} \to G$
be the corresponding $\mathcal O$-group scheme morphism.
Let $Z=ker(\pi)$.
It is known that
$Z$
is contained in the center $Cent(G^{sc}_{})$ of
$G^{sc}_{}$ and
$Z$
is a finite group scheme of multiplicative type.
It is known that
$G_{}=G^{sc}_{}/Z$
and $\pi$ is finite surjective and strictly flat.
Thus the sequence of $\mathcal O$-group schemes
\begin{equation}
\label{ZenterCsDer}
\{1\} \to Z \xra{i} G^{sc}_{} \xra{\pi} G_{} \to \{1\},
\end{equation}
gives rise to an exact sequence of pointed sets
\begin{equation}
\label{H1OfZenterCsDer1}
H^1_{\text{fppt}}(\mathcal O, Z)/\partial (G_{}(\mathcal O)) \to H^1_{\text{fppt}}(\mathcal O,G^{sc}_{}) \to H^1_{\text{fppt}}(\mathcal O,G_{})
\to H^2_{\text{fppt}}(\mathcal O,Z)
\end{equation}

By the Theorem
\ref{PurityForSubgroup}
the functor
$$
{\cal F}(R)= \textrm{H}^1_{\text{fppt}}(R,Z)/ Im(\delta_{\pi,R}) = \textrm{H}^1_{\text{fppt}}(R,Z)/\partial (G_{}(R)).
$$
satisfies purity for the ring $\mathcal O$.
The following result is known (see \cite[Thm.11.7]{Gr3})
\begin{lem}
\label{FlatAndEtale}
Let $R$ be a noetherian ring. Then for a reductive $R$-group scheme $H$ and for $n=0, 1$ the canonical map
$\textrm{H}^n_{\text{\'et}}(R, H) \to \textrm{H}^n_{\text{fppt}}(R, H)$
is a bijection of pointed sets. For a $R$-tori $T$ and for each integer $n \geq 0$
the canonical map
$\textrm{H}^n_{\text{\'et}}(R, T) \to \textrm{H}^n_{\text{fppt}}(R, T)$
is an isomorphism.
\end{lem}

\begin{lem}
\label{Gr_Serre_For_H2_Z}
For the ring $\mathcal O$ above one has
$ker[\textrm{H}^2_{\text{fppt}}(\mathcal O, Z) \to \textrm{H}^2_{\text{fppt}}(K, Z)]= *$.
\end{lem}

\begin{proof}
In fact, consider the closed embedding
$Z \hra R_{A/{\mathcal O}}(\mathbb G_{m,\ A})$
into the permutation torus
$T^{+}=R_{A/\mathcal O}(\mathbb G_{m,\ A})$
from Theorem
\ref{PurityForSubgroup}.
Set
$T= T^{+}/Z$.
The short sequence of $\mathcal O$-group schemes
$1 \to Z \to T^{+} \to T \to 1$
gives rise to a short exact sequence of group sheaves in the $\text{fppt}$-topology.
That sequence in turn gives rise to a long exact sequence of cohomology groups
\begin{equation}
\label{H2ZandBr}
\dots \to \textrm{H}^1_{\text{fppt}}(\mathcal O, T^{+}) \to \textrm{H}^1_{\text{fppt}}(\mathcal O, T)
\xra{\partial} \textrm{H}^2_{\text{fppt}}(\mathcal O, Z) \to \textrm{H}^2_{\text{fppt}}(\mathcal O, T^{+})
\end{equation}
Firstly note that
$\textrm{H}^1_{\text{fppt}}(\mathcal O, T^{+})=\textrm{H}^1_{\text{\'et}}(\mathcal O, T^{+})= \textrm{H}^1_{\text{\'et}}(A, \Bbb G_{m,A})=0$,
since $A$ is semi-local.
Secondly note that
$\textrm{H}^2_{\text{fppt}}(\mathcal O, T^{+})=\textrm{H}^2_{\text{\'et}}(\mathcal O, T^{+})= \textrm{H}^2_{\text{\'et}}(A, \Bbb G_{m,A})$.
The map
$\textrm{H}^1_{\text{fppt}}(\mathcal O, T) \to \textrm{H}^1_{\text{fppt}}(K, T)$
is injective by Lemma
\ref{FlatAndEtale}
and
\cite{C-T-S}.
The homomorphism
$\textrm{H}^2_{\text{fppt}}(\mathcal O, T^{+}) \to \textrm{H}^2_{\text{fppt}}(K, T^{+})$
coincides with the map
$\textrm{H}^2_{\text{\'et}}(A, \Bbb G_{m,A}) \to \textrm{H}^2_{\text{\'et}}(A \otimes_{\mathcal O} K, \Bbb G_{m,A \otimes_{\mathcal O} K})$.
The latter map is injective, since
$A$ is regular semi-local of geometric type
(see \cite[??]{Gr2}). Now a diagram chaise completes the proof of the Lemma.
\end{proof}
Continue the proof of the equality (\ref{Gr_SerreForGsemisimple}).
We have the exact sequence of pointed sets
\begin{equation}
\label{H1OfZenterCsDer2}
H^1_{\text{fppt}}(\mathcal O, Z)/\partial (G_{}(\mathcal O)) \to H^1_{\text{fppt}}(\mathcal O,G^{sc}_{}) \to H^1_{\text{fppt}}(\mathcal O,G_{})
\to H^2_{\text{fppt}}(\mathcal O,Z)
\end{equation}
and furthermore a commutative diagram
with exact arrows
\begin{equation}
\label{RactangelDiagram}
    \xymatrix{
H^1_{\text{fppt}}(\mathcal O, Z)/\partial (G_{}(\mathcal O)) \ \ \ar[r]^{i_*}\ar[d]_{\alpha} & H^1_{\text{fppt}}(\mathcal O,G^{sc}_{}) \ar[r]^{\pi_*} \ar[d]^{\beta} &
H^1_{\text{fppt}}(\mathcal O,G_{}) \ar[r]^{\partial} \ar[d]_{\gamma} & H^2_{\text{fppt}}(\mathcal O,Z) \ar[d]^{\delta} &\\
H^1_{\text{fppt}}(\mathcal O_{\mathfrak p}, Z)/\partial (G_{}(\mathcal O_{\mathfrak p})) \ \ \ar[r]^{i^{\prime}_*}\ar[d]_{\alpha_{\mathfrak p}} & H^1_{\text{fppt}}(\mathcal O_{\mathfrak p},G^{sc}_{}) \ar[r]^{\pi_*} \ar[d]^{\beta_{\mathfrak p}} &
H^1_{\text{fppt}}(\mathcal O_{\mathfrak p},G_{}) \ar[r]^{\partial} \ar[d]_{\gamma_{\mathfrak p}} & H^2_{\text{fppt}}(\mathcal O_{\mathfrak p},Z) \ar[d]^{\delta_{\mathfrak p}} &\\
H^1_{\text{fppt}}(K, Z)/\partial (G_{}(K)) \ \ \ar[r]^{i^{\prime\prime}_*} & H^1_{\text{fppt}}(K,G^{sc}_{}) \ar[r]^{\pi_*} &
H^1_{\text{fppt}}(K,G_{}) \ar[r]^{\partial}  & H^1_{\text{fppt}}(K,Z) &\\
}
\end{equation}
Here
$\mathfrak p \subset \mathcal O$ is a hight one prime ideal in $\mathcal O$.
The maps $i_*$, $i^{\prime}_*$ and $i^{\prime\prime}_*$ are injective
(compare (\cite[Ch.I,Sect.5, Prop.39 and Cor.1 of Prop.40]{Se})).
Set
$\alpha_K=\alpha_{\mathfrak p} \circ \alpha$,
$\beta_K= \beta_{\mathfrak p} \circ \beta$,
$\gamma_K = \gamma_{\mathfrak p} \circ \gamma$,
$\delta_K = \delta_{\mathfrak p} \circ \delta$.
By a theorem of Nisnevich
\cite{Ni}
and Lemma
\ref{FlatAndEtale}
one has
\begin{equation}
\label{NisnevichForPurity}
ker(\beta_{\mathfrak p})=ker(\gamma_{\mathfrak p})=* \ .
\end{equation}
Thus $ker(\alpha_{\mathfrak p})=*$.
By the assumptions of Theorem \ref{MainThmGeometric}
and by Lemma
\ref{FlatAndEtale}
one has
\begin{equation}
\label{Gr_SerreForG_sc}
ker[H^1_{\text{fppt}}(\mathcal O,G^{sc}_{}) \to H^1_{\text{fppt}}(K,G^{sc}_{})]=* \ .
\end{equation}
By Lemma
\ref{Gr_Serre_For_H2_Z}
the map
$\delta_K$
is injective.
As mentioned right above Lemma
\ref{FlatAndEtale}
the functor
$
{\cal F}(R)= \textrm{H}^1_{\text{fppt}}(R,Z)/ \partial (G_{}(R))
$
satisfies purity for the ring $\mathcal O$.
Now we are ready to make a diagram chaise.

Let
$\xi \in ker(\gamma_K)$,
then
$\partial(\xi) \in ker(\delta_K)$.
By
\cite{C-T-S}
one has
$ker(\delta_K)=*$,
whence
$\partial(\xi)=*$
and
$\xi=\pi_*(\zeta)$
for an
$\zeta \in H^1_{\text{\'et}}(\mathcal O,G^{sc}_{})$.
Since
$\gamma_K(\xi)=*$
and
$ker(\gamma_{\mathfrak p})=*$
we see that
$\gamma(\xi)=*$.
Thus
$\pi_*(\beta(\zeta))=*$
and
$\beta(\zeta)= i^{\prime}_*(\epsilon_{\mathfrak p})$
for an
$\epsilon_{\mathfrak p} \in H^1_{\text{fppt}}(\mathcal O_{\mathfrak p}, Z)/\partial (G_{}(\mathcal O_{\mathfrak p}))$.
A diagram chaise shows that there exists a unique element
$\epsilon_K \in H^1_{\text{fppt}}(K, Z)/\partial (G_{}(K))$
such that
for each hight one prime ideal
$\mathfrak p$ of $\mathcal O$
one has
$$\alpha_{\mathfrak p}(\epsilon_{\mathfrak p})= \epsilon_K \in H^1_{\text{fppt}}(K, Z)/\partial (G_{}(K)).$$
By Purity Theorem
\ref{PurityForSubgroup}
there exists an element
$\epsilon \in H^1_{\text{fppt}}(\mathcal O, Z)/\partial (G_{}(\mathcal O))$
such that
$\alpha_K(\epsilon)=\epsilon_K$.
The element
$\epsilon_K$
has the property that
$i^{\prime\prime}_*(\epsilon_K)= \beta_K(\zeta)$.
Whence
$\beta_K(i_*(\epsilon))=\beta_K(\zeta)$.
The map
$\beta_K: H^1_{\text{\'{e}t}}(\mathcal O,G^{sc}_{}) \to H^1_{\text{\'{e}t}}(K,G^{sc}_{})$
is injective since by the hypotheses of Theorem
\ref{MainThmGeometric}
such a map has trivial kernel for all
semi-simple simply-connected reductive $\mathcal O$-group schemes.
Whence
$i_*(\epsilon) = \zeta$
and
$\xi=i_*(\pi_*(\epsilon))=*$.
The semi-simple case of Theorem \ref{MainThmGeometric} is proved.
\end{proof}

\begin{clm}
\label{semisimpleinjective}
Under the hypotheses of Theorem \ref{MainThmGeometric} for all semi-simple reductive $\mathcal O$-group scheme $G$
the map
$H^1_{\text{\'et}}(\mathcal O,G) \to H^1_{\text{\'et}}(K,G)$
is injective.
\end{clm}
In fact, let $\xi, \zeta \in H^1_{\text{\'et}}(\mathcal O,G)$ be two elements such that its images $\xi_K, \zeta_K$ in
$H^1_{\text{\'et}}(K,G)$
are equal. Let $_{\xi}G$, $_{\zeta}G$ be the corresponding principal $G$-bundles over $\mathcal O$ and $G(\zeta)$ be
the inner form of the $\mathcal O$-group scheme $G$ corresponding to $\zeta$. The $\mathcal O$-scheme
$\underline {Iso}(_{\xi}G, _{\zeta}G)$
is a principal
$G(\zeta)$-bundle over $\mathcal O$,
which is trivial over $K$. Since $G(\zeta)$ is semi-simple reductive over $\mathcal O$,
the $\mathcal O$-scheme
$\underline {Iso}(_{\xi}G, _{\zeta}G)$
has an $\mathcal O$-point. Whence the Claim.

\begin{proof}[Proof of Theorem \ref{MainThmGeometric}]
Let $\mathcal O$ and $G$ be the same as in Theorem \ref{MainThmGeometric}.
Consider a short sequence of reductive $\mathcal O$-group schemes
\begin{equation}
\label{DerRedTori}
\{1\} \to G_{der} \xra{i} G \xra{\mu} C \to \{1\},
\end{equation}
where
$G_{der}$ is the derived $\mathcal O$-group scheme of $G$ and $C=corad(G)$ be a tori over $\mathcal O$ and
$\mu=f_0$ (see
\cite[Exp.XXII, Thm.6.2.1]{D-G}).
By that Theorem the morphism $\mu$ is smooth and its kernel is the reductive $\mathcal O$-group scheme $G_{der}$.
Moreover
$G_{der}$
is a semi-simple $\mathcal O$-group scheme.
By Claim \ref{semisimpleinjective} the map
\begin{equation}
\label{Gr_SerreForG_der}
H^1_{\text{\'et}}(\mathcal O,G_{der}) \to H^1_{\text{\'et}}(K,G_{der})
\end{equation}
is injective.
We need to prove that
\begin{equation}
\label{Gr_SerreForG}
ker[H^1_{\text{\'et}}(\mathcal O,G) \to H^1_{\text{\'et}}(K,G)]=* .
\end{equation}




The sequence
(\ref{DerRedTori})
of $\mathcal O$-group schemes gives a short exact sequence of the corresponding
sheaves in the \'{e}tale topology on the big \'{e}tale site. That sequence of sheaves
gives rise to a commutative diagram with exact arrows of pointed sets

\begin{equation}
\label{RactangelDiagram}
    \xymatrix{
\{1\} \ar[r]  & C(\mathcal O)/\mu(G(\mathcal O)) \ar[r]^{\partial}\ar[d]_{\alpha} & H^1_{\text{\'et}}(\mathcal O,G_{der}) \ar[r]^{i_*} \ar[d]^{\beta} &
H^1_{\text{\'et}}(\mathcal O,G) \ar[r]^{\mu} \ar[d]_{\gamma} & H^1_{\text{\'et}}(\mathcal O,C) \ar[d]^{\delta} &\\
\{1\}  \ar[r]  & C(\mathcal O_{\mathfrak p})/\mu(G(\mathcal O_{\mathfrak p})) \ar[r]^{\partial}\ar[d]_{\alpha_{\mathfrak p}} & H^1_{\text{\'et}}(\mathcal O_{\mathfrak p},G_{der}) \ar[r]^{i_*} \ar[d]^{\beta_{\mathfrak p}} &
H^1_{\text{\'et}}(\mathcal O_{\mathfrak p},G) \ar[r]^{\mu} \ar[d]_{\gamma_{\mathfrak p}} & H^1_{\text{\'et}}(\mathcal O_{\mathfrak p},C) \ar[d]^{\delta_{\mathfrak p}} &\\
\{1\}  \ar[r] & C(K)/\mu(G(K)) \ar[r]^{\partial} & H^1_{\text{\'et}}(K,G_{der}) \ar[r]^{i_*} &
H^1_{\text{\'et}}(K,G) \ar[r]^{\mu}  & H^1_{\text{\'et}}(K,C) &\\
}
\end{equation}
Here
$\mathfrak p \subset \mathcal O$ is a hight one prime ideal in $\mathcal O$.
Set
$\alpha_K=\alpha_{\mathfrak p} \circ \alpha$,
$\beta_K= \beta_{\mathfrak p} \circ \beta$,
$\gamma_K = \gamma_{\mathfrak p} \circ \gamma$,
$\delta_K = \delta_{\mathfrak p} \circ \delta$.
By a theorem of Nisnevich
\cite{Ni}
one has
\begin{equation}
\label{NisnevichForPurity}
ker(\alpha_{\mathfrak p})=ker(\beta_{\mathfrak p})=ker(\gamma_{\mathfrak p})=*
\end{equation}
Let
$\xi \in ker(\gamma_K)$,
then
$\mu(\xi) \in ker(\delta_K)$.
By
\cite{C-T-S}
one has
$ker(\delta_K)=*$,
whence
$\mu(\xi)=*$
and
$\xi=i_*(\zeta)$
for an
$\zeta \in H^1_{\text{\'et}}(\mathcal O,G_{der})$.
Since
$\gamma_K(\xi)=*$
and
$ker(\gamma_{\mathfrak p})=*$
we see that
$\gamma(\xi)=*$.
Whence
$i_*(\beta(\zeta))=*$
and
$\beta(\zeta)= \partial (\epsilon_{\mathfrak p})$
for an
$\epsilon_{\mathfrak p} \in C(\mathcal O_{\mathfrak p})/\mu(G(\mathcal O_{\mathfrak p}))$.
A diagram chaise
AND Lemma \ref{BoundaryInj}
show that there exists a unique element
$\epsilon_K \in C(K)/\mu(G(K))$
such that
for each hight one prime ideal
$\mathfrak p$ of $\mathcal O$
one has
$\alpha_{\mathfrak p}(\epsilon_{\mathfrak p})= \epsilon_K \in C(K)/\mu(G(K))$.
By Purity Theorem (Theorem \ref{TheoremA})
there exists an element
$\epsilon \in C(\mathcal O)/\mu(G(\mathcal O))$
such that
$\alpha_K(\epsilon)=\epsilon_K$.
The element
$\epsilon_K$
has the property that
$\partial (\epsilon_K)= \beta_K(\zeta)$.
Whence
$\beta_K(\partial (\epsilon))=\beta_K(\zeta)$.
The map $\beta_K$ is injective as indicated in the beginning
of the proof.
Whence
$\partial (\epsilon) = \zeta$
and
$\xi=i_*(\partial(\epsilon))=*$.
The proof of Theorem
\ref{MainThmGeometric} is completed.

\end{proof}

\section{Examples}
\label{ExamplesOfPurityResults}
We follow here the notation of The Book of Involutions
\cite{KMRT}.
The field $k$ is of characteristic different from $2$, $3$ and $5$.
We added the latter
requirement to be sure that all the algebraic groups in this section are reductive.
The functors
(\ref{Example1})
to
(\ref{Example12})
satisfy purity for regular local rings containing $k$
as follows either from Theorem
\ref{PurityForSubgroup}
or from
Theorem (A).

\begin{itemize}
\item[(1)]
Let $G$ be a simple algebraic group over the field $k$,
$Z$ a central subgroup, $G'=G/Z$, $\pi: G \to G'$
 the canonical morphism.
For any $k$-algebra $A$ let
$\delta_{\pi,R}: G'(R) \to \textrm{H}^1_{\text{\'et}}(R,Z)$
be the boundary operator. One has a functor
\begin{equation}
\label{Example1}
R \mapsto \textrm{H}^1_{\text{\'et}}(R,Z)/ \textrm{Im}(\delta_{\pi,R}).
\end{equation}

\item[(2)]
Let $(A,\sigma)$ be a finite separable $k$-algebra with an orthogonal involution.
Let
$\pi: \textrm{Spin}(A, \sigma) \to \textrm{PGO}^+(A, \sigma)$
be the canonical morphism of the spinor $k$-group scheme to
the projective orthogonal $k$-group scheme. Let
$Z=ker (\pi)$.
For a $k$-algebra $R$ let
$\delta_R: \textrm{PGO}^+(A, \sigma)(R) \to \textrm{H}^1_{\text{\'et}}(R,Z)$
be the boundary operator.
One has a functor
\begin{equation}
\label{Example2}
R \mapsto \textrm{H}^1_{\text{\'et}}(R,Z)/Im(\delta_R).
\end{equation}
In $(2a)$ and $(2b)$ below we describe this functor somewhat more explicitly
following
\cite{KMRT}.

\item[(2a)]
Let $C(A, \sigma)$ be the Clifford algebra. Its center $l$ is an \'etale
quadratic $k$-algebra.
Assume that $deg(A)$ is divisible by $4$.
Let
$\Omega(A,\sigma)$ be the extended Clifford group
\cite[Definition given just below (13.19)]{KMRT}.
Let
$\underline \sigma$
be the canonical involution of
$C(A,\sigma)$
as it is described in
\cite[just above (8.11)]{KMRT}.
Then
$\underline \sigma$
is either orthogonal or symplectic by
\cite[Prop.8.12]{KMRT}.
Let
$
\underline \mu: \Omega(A,\sigma) \to R_{l/k}(\mathbb G_{m,l})
$
be the multiplier map defined in
\cite[just above (13.25)]{KMRT}
by
$\underline \mu (\omega)= \underline \sigma(\omega)\cdot \omega$.
Set
$R_l=R \otimes_k l$.
For a field or a local ring $R$ one has
$
\textrm{H}^1_{\text{\'et}}(R,Z)/Im(\delta_R)=R_l^{\times}/\underline \mu (\Omega(A,\sigma)(R))
$
by
\cite[the diagram in (13.32)]{KMRT}.
Consider the functor
\begin{equation}
\label{Example2a}
R \mapsto R_l^{\times}/\underline \mu (\Omega(A,\sigma)(R)).
\end{equation}
It coincides with the functor
$R \mapsto \textrm{H}^1_{\text{\'et}}(R,Z)/Im(\delta_R)$
on local rings containing $k$.

\item[(2b)]
Now let $deg(A)=2m$ with odd $m$. Let $\tau: l \to l$ be the involution of $l/k$.
The kernel of the morphism
$R_{l/k}(\mathbb G_{m,l}) \xra{id - \tau} R_{l/k}(\mathbb G_{m,l})$
coincides with $\mathbb G_{m,k}$. Thus $id-\tau$ 
induces a $k$-group scheme morphism
which we denote
$\overline {id-\tau}: R_{l/k}(\mathbb G_{m,l})/\mathbb G_{m,k} \hra R_{l/k}(\mathbb G_{m,l})$.
Let
$
\underline \mu: \Omega(A,\sigma) \to R_{l/k}(\mathbb G_{m,l})
$
be the multiplier map defined in
\cite[just above (13.25)]{KMRT}
by
$\underline \mu (\omega)= \underline \sigma(\omega)\cdot \omega$.
Let
$\kappa: \Omega(A,\sigma) \to R_{l/k}(\mathbb G_{m,l})/\mathbb G_{m,k}$
be the $k$-group scheme morphism described in
\cite[Prop.13.21]{KMRT}.
The composition
$\overline {(id-\tau)} \circ \kappa$
lands in
$R_{l/k}(\mathbb G_{m,l})$.
Let
$U \subset \mathbb G_{m,k} \times R_{l/k}(\mathbb G_{m,l})$
be a closed $k$-subgroup consisting of all $(\alpha, z)$ such that
$\alpha^4 = N_{l/k}(z)$.

Set
$\mu_*=(\underline \mu, [\overline {(id-\tau)} \circ \kappa]\cdot\underline \mu2): \Omega(A,\sigma)
\to \mathbb G_{m,k} \times R_{l/k}(\mathbb G_{m,l})$.
This $k$-group scheme morphism lands in $\textrm{U}$.
So we get a $k$-group scheme morphism
$\mu_*: \Omega(A,\sigma) \to \textrm{U}$.
On the level of $k$-rational points it coincides with the one described in
\cite[just above (13.35)]{KMRT}. For a field or a local ring one has
$$
\textrm{H}^1_{\text{\'et}}(R,Z)/Im(\delta_R)=
\textrm{U}(R)/[\{(N_{l/k}(\alpha), \alpha^4) | \alpha \in R_l^{\times}\}\cdot \mu_*(\Omega(A,\sigma)(R))].
$$
Consider the the functor
\begin{equation}
\label{Example2b}
R \mapsto \textrm{U}(R)/[\{(N_{l/k}(\alpha), \alpha^4) | \alpha \in R_l^{\times}\}\cdot \mu_*(\Omega(A,\sigma)(R))].
\end{equation}
It coincides with the functor
$R \mapsto \textrm{H}^1_{\text{\'et}}(R,Z)/Im(\delta_R)$
on local rings containing $k$.

\item[(3)]
Let
$\Gamma(A,\sigma)$
be the Clifford group $k$-scheme of $(A,\sigma)$.
Let
$\textrm{Sn}: \Gamma(A,\sigma) \to \mathbb G_{m,k}$
be the spinor norm map. It is dominant.
Consider the functor
\begin{equation}
\label{Example3}
R \mapsto R^{\times}/\textrm{Sn}(\Gamma(A,\sigma)(R)).
\end{equation}
Purity for this functor was originally  proved in
\cite[Thm.3.1]{Z}.
In fact,
$\Gamma(A,\sigma)$
is $k$-rational.

\item[(4)]
We follow here the Book of Involutions
\cite[\S 23]{KMRT}.
Let $A$ be a separable finite dimensional $k$-algebra with center $l$
and $k$-involution $\sigma$ such that $k$ coincides with all $\sigma$-invariant
elements of $l$, that is $k=l^{\sigma}$.
Consider the $k$-group schemes of
similitudes
of $(A,\sigma)$:
$$
\textrm{Sim}(A,\sigma)(R)=\{a \in A_R^{\times}\ | a\cdot \sigma_R (a) \in l_K^{\times} \}.
$$
We have a $k$-group scheme morphism
$
\mu: \textrm{Sim}(A,\sigma) \to \mathbb G_{m,k}, \ a \mapsto a \cdot \sigma (a).
$
It gives an exact sequence of algebraic $k$-groups
$$
\{1\} \to \textrm{Iso}(A,\sigma) \to \textrm{Sim}(A,\sigma) \to \mathbb G_{m,k} \to \{1\}.
$$
One has a the functor
\begin{equation}
\label{Example4}
R \mapsto R^{\times}/\mu (\textrm{Sim}(A,\sigma)(R)).
\end{equation}
Purity for this functor was originally  proved in
\cite[Thm.1.2]{Pa}.
Various particular cases are obtained
considering unitary, symplectic and orthogonal involutions.
\item[(4a)]
In the case of an orthogonal involution $\sigma$ the connected component
$\textrm{GO}^+(A, \sigma)$
\cite[(12.24)]{KMRT}
of the similitude $k$-group scheme
$\textrm{GO}(A, \sigma):=\textrm{Sim}(A,\sigma)$
has the index two in
$\textrm{GO}(A, \sigma)$.
The restriction of $\mu$ to
$\textrm{GO}^+(A, \sigma)$
is still a dominant morphism to
$\mathbb G_{m,k}$. One has a functor
\begin{equation}
\label{Example4a}
R \mapsto R^{\times}/\mu (\textrm{GO}^+(A,\sigma)(R))
\end{equation}
It seems that its purity
does not follow from
\cite[Thm.1.2]{Pa}. In fact we do not know
whether the norm principle holds for
$\mu: \textrm{GO}^+(A, \sigma) \to \mathbb G_{m,k}$
or not.

\item[(5)]
Let $A$ be a central simple algebra (csa) over $k$ and
$\textrm{Nrd}:GL_{1,A} \to \mathbb G_{m,k}$
 the reduced norm morphism. One has a functor
\begin{equation}
\label{Example5}
R \mapsto R^{\times}/\textrm{Nrd}(\textrm{GL}_{1,A}(R)).
\end{equation}
Purity for this functor was originally  proved in
\cite[Thm.5.2]{C-TO}.
\item[(6)]
Let $(A,\sigma)$ be a finite separable $k$-algebra with
a unitary involution such that its center $l$ is a quadratic extension of $k$.
Let $U(A,\sigma)$ be the unitary $k$-group scheme. Let
$U_l(1)$ be an algebraic tori given by $N_{l/k}=1$.
One has a functor
\begin{equation}
\label{Example6}
R \mapsto \textrm{U}_l(1)(R)/\textrm{Nrd}(\textrm{U}_{A,\sigma}(R)) =
\{\alpha \in R^{\times}_l | N_{l/k}(\alpha)=1 \} /\textrm{Nrd}(\textrm{U}_{A,\sigma}(R))
\end{equation}
where
$\textrm{Nrd}$
is the reduced norm map.
Purity for this functor was originally  proved in
\cite[Thm.3.3]{Z}.

\item[(7)]
Let $A$ be a csa of degree $3$ over $k$, $\textrm{Nrd}$  the reduced norm
and $\textrm{Trd}$ be the reduced trace. Consider the cubic form
on the $27$-dimensional $k$-vector space
$A \times A \times A$
given by
$N:= \textrm{Nrd}(x)+\textrm{Nrd}(y)+\textrm{Nrd}(z)-\textrm{Trd}(xyz)$.
Let
$\textrm{Iso}(A,N)$
be the $k$-group scheme of isometries of $N$ and
$\textrm{Sim}(N)$
be the $k$-group scheme of similitudes of $N$.
It is known that
$\textrm{Iso}(N)$ is a normal algebraic subgroup in $\textrm{Sim}(A,N)$
and the factor group coincides with
$\mathbb G_{m,k}$. So we have a canonical $k$-group morphism
(the multiplier)
$\mu: \textrm{Sim}(N) \to \mathbb G_{m,k}$.
Now one has a functor
\begin{equation}
\label{Example11}
R \mapsto R^{\times}/\mu(\textrm{Sim}(N)(R)).
\end{equation}
Note that the connected component of $\textrm{Iso}(N)$ is
a simply connected algebraic $k$-group of the type $E_6$.

\item[(8)]
Let $(A,\sigma)$ be a csa of degree $8$ over $k$
with a symplectic involution. Let $V \subset A$ be the subspace
of all skew-symmetric elements. It is of dimension $28$.
Let
$\textrm{Pfd}$
be the reduced Pfaffian on $V$ and
$\textrm{Trd}$
be the reduced trace on $A$. Consider the degree $4$ form on
the space
$V \times V$
given by
$F:=\textrm{Pfr}(x)+\textrm{Pfr}(y)-1/4\textrm{Trd}((xy)2)-1/16{\textrm{Trd}(xy)}^2$.
Consider the symplectic form on
$V \times V$
given by
$\phi((x_1,y_1),(x_2,y_2))= \textrm{Trd}(x_1y_2-x_2y_1)$.
Let
$\textrm{Iso}(F)$
(resp. $\textrm{Iso}(\phi)$)
be the $k$-group scheme of isometries of the pair $F$
(resp. of $\phi)$).
Let
$\textrm{Sim}(F)$
(resp. $\textrm{Sim}(\phi)$)
be the $k$-group scheme of similitudes of $F$
(resp. of $\phi$).
Set
$G=\textrm{Iso}(F) \cap \textrm{Iso}(\phi)$
and
$G^+=\textrm{Sim}(F) \cap \textrm{Sim}(\phi)$.
It is known that
$G$ is a normal algebraic subgroup in $G^+$
and the factor group coincides with
$\mathbb G_{m,k}$. So we have a canonical $k$-group morphism 
$\mu: G^+ \to \mathbb G_{m,k}$.
Now one has a functor
\begin{equation}
\label{Example12}
R \mapsto R^{\times}/\mu(G^+(R)).
\end{equation}
Note that $G$ is a simply-connected group of the type $E_7$.

\end{itemize}

\section{Appendix}
Theorem \ref{PropEquatingGroups} is proved in the present Section.
\begin{thm}
\label{PropEquatingGroups2}
Let $S$ be a regular semi-local
irreducible scheme.
Assume that all the closed points of $S$ have {\bf finite residue fields}.
Assume
that $G_1$ and $G_2$ are reductive $S$-group schemes
which are forms of each other.
Let $T \subset S$ be a connected non-empty closed
sub-scheme of $S$, and
$\varphi: G_1|_T \to G_2|_T$
be $S$-group scheme isomorphism.
Then there exists a finite \'{e}tale morphism
$\tilde S \xra{\pi} S$ together with its section $\delta: T \to
\tilde S$ over $T$ and $\tilde S$-group scheme isomorphisms
$\Phi: \pi^*{G_1} \to \pi^*{G_2}$
such that
\begin{itemize}
\item[(i)]
$\delta^*(\Phi)=\varphi$,
\item[(ii)]
the scheme $\tilde S$ is irreducible.
\end{itemize}
\end{thm}

\begin{prop}
\label{PropEquatingGroups3}
Theorem \ref{PropEquatingGroups2} holds in the case when the groups $G_1$ and $G_2$ are semi-simple
(see \cite[Prop.5.1]{Pan1}).
\end{prop}

\begin{prop}
\label{PropEquatingGroups4}
Theorem \ref{PropEquatingGroups2} holds in the case when the groups $G_1$ and $G_2$ are tori
and, more generally, in the case when the groups $G_1$ and $G_2$ are
of multiplicative type.
\end{prop}

\begin{prop}
\label{PropEquatingGroups5}
Let $T$ and $S$ be the same as in Theorem \ref{PropEquatingGroups2}.
Let $M_1$ and $M_2$ be two $S$-group schemes of multiplicative type.
Let $\alpha_1, \alpha_2: M_1 \rightrightarrows M_2$
be two $S$-group scheme morphisms such that
$\alpha_1|_T = \alpha_2|_T$.
Then $\alpha_1 = \alpha_2$.
\end{prop}

\begin{proof}[Proof of Theorem \ref{PropEquatingGroups2}]
Let $Rad(G_r) \subset G_r$ be the radical of $G_r$ and let
$der(G_r) \subset G_r$ be the derived subgroup of $G_r$ ($r=1,2$)
(see \cite[Exp.XXII, 4.3]{D-G}).
By the very definition the radical is a tori. The $S$-group scheme
$der(G_r)$ is semi-simple ($r=1,2$).
Set $Z_r:= Rad(G_r) \cap der (G_r)$.
The above embeddings induce natural $S$-group morphisms
$$\Pi_r: Rad(G_r) \times_S der(G_r) \to G_r$$
with $Z_r$ as the kernel ($r=1,2$).
By \cite[Exp.XXII,Prop.6.2.4]{D-G} $\Pi_r$ is a central isogeny.
Particularly, $\Pi_r$ is a faithfully flat finite morphism by
\cite[Exp.XXII,Defn.4.2.9]{D-G}.
Let
$i_r: Z_r \hookrightarrow Rad(G_r) \times_S der(G_r)$
be the closed embedding.

The $T$-group scheme isomorphism
$\varphi: G_1|_T \to G_2|T$
induces certain $T$-group scheme isomorphisms
$\varphi_{der}: der(G_1|_T) \to der(G_2|_T)$,
$\varphi_{rad}: rad(G_1|_T) \to rad(G_2|_T)$
and
$\varphi_Z: Z_1|_T \to Z_2|_T$
{\bf such that}
$$(\Pi_2)|_T \circ (\varphi_{der}\times \varphi_{rad})=\varphi \circ (\Pi_1)|_T \ \text{and} \
i_{2,T} \circ \varphi_Z = (\varphi_{rad} \times \varphi_{der}) \circ i_{1,T}.$$

By Propositions
\ref{PropEquatingGroups3}
and
\ref{PropEquatingGroups4}
there exist a finite \'{e}tale morphism
$\pi: \tilde S \to S$
(with an irreducible scheme $\tilde S$)
and its section
$\delta: T \to \tilde S$
over $T$ and $\tilde S$-group scheme isomorphisms
$$\Phi_{der}: der(G_{1, \tilde S}) \to der(G_{2,\tilde S}) \ \ \text{,} \ \
\Phi_{Rad}: Rad(G_{1, \tilde S}) \to Rad(G_{2,\tilde S}) \ \ \text{and} \ \
\Phi_Z: Z_{1, \tilde S} \to Z_{2,\tilde S}$$
such that
$\delta^*(\Phi_{der})= \varphi_{der}$,
$\delta^*(\Phi_{rad})= \varphi_{rad}$
and
$\delta^*(\Phi_Z)= \varphi_Z$.

Since $Z_r$ is contained in the center of $der(G_r)$ and is of multiplicative type
Proposition
\ref{PropEquatingGroups5}
yields the equality
$$i_{2, \tilde S} \circ \Phi_Z = (\Phi_{Rad} \times \Phi_{der}) \circ i_{1, \tilde S}:
Z_{1, \tilde S} \to Rad(G_{2, \tilde S}) \times_{\tilde S} der(G_{2, \tilde S}).$$
Thus $(\Phi_{Rad} \times \Phi_{der})$ induces an $\tilde S$-group scheme isomorphism
$$\Phi: G_{1, \tilde S} \to G_{2, \tilde S}$$
such that
$\Pi_{2, \tilde S} \circ (\Phi_{Rad} \times \Phi_{der})= \Phi \circ \Pi_{1, \tilde S}$.
The latter equality yields the following one
$(\Pi_2)|_T \circ \delta^*(\Phi_{Rad} \times \Phi_{der})= \delta^*(\Phi) \circ (\Pi_1)|_T$,
which in turn yields the equality
$$(\Pi_2)|_T \circ (\varphi_{rad} \times \varphi_{der})= \delta^*(\Phi) \circ (\Pi_1)|_T.$$
Comparing it with the equality
$(\Pi_2)|_T \circ (\varphi_{rad} \times \varphi_{der})= \varphi \circ (\Pi_1)|_T$
and using the fact that
$(\Pi_1)|_T$
is strictly flat we conclude the equality
$\delta^*(\Phi)= \varphi$.

\end{proof}

\begin{proof}[Proof of Theorem \ref{PropEquatingGroups}]
By Theorem
\ref{PropEquatingGroups2}
there exists a finite \'{e}tale morphism
$\tilde S \xra{\pi} S$ together with its section $\delta: T \to
\tilde S$ over $T$ and $\tilde S$-group scheme isomorphisms
$$\Phi: G_{1, \tilde S} \to G_{2, \tilde S} \ \ \text{and} \
\Psi: C_{1, \tilde S} \to G_{1, \tilde S}$$
such that
$\delta^*(\Phi)= \varphi$, $\delta^*(\Psi)= \psi$.
It remains to show that
$\mu_{2, \tilde S} \circ \Phi = \Psi \circ \mu_{1, \tilde S}$.
To prove this equality recall that
$\mu_r$ can be naturally presented as a composition
$$G_r \xrightarrow{can_r} Corad(G_r) \xrightarrow{{\bar \mu}_r} C_r.$$
Since
$can_{2, \tilde S} \circ \Phi = Corad(\Phi) \circ can_{1, \tilde S}$
it remains to check that
${\bar \mu}_{2, \tilde S} \circ Corad(\Phi)=\Psi \circ {\bar \mu}_{1, \tilde S}$.
The latter equality follows from Proposition
\ref{PropEquatingGroups5}
and the equality
$({\bar \mu_2}|_{T}) \circ Corad(\varphi) = \psi \circ ({\bar \mu_1}|_{T})$,
which holds since
$(\mu_2|_{T}) \circ \varphi = \psi \circ (\mu_1|_{T})$
by the very assumption of the Theorem.

\end{proof}

Let $k$ be {\bf a finite field}.
Let $U$ be as in Definition
\ref{DefnNiceTriple}.
Let $S^{\prime}$ be an irreducible regular semi-local scheme over $k$ and
$p: S^{\prime} \to U$ be a $k$-morphism.
Let $T^{\prime}\subset S^{\prime}$ be a closed sub-scheme of $S^{\prime}$
such that the restriction
$p|_{T^{\prime}}: T^{\prime} \to U$ is an isomorphism.
We will assume below that $dim(T^{\prime}) < dim(S^{\prime})$,
where $dim$ is the Krull dimension.
For any closed point $u \in U$ and any $U$-scheme $V$ let $V_u=u\times_U V$ be
the fibre of the scheme $V$ over the point $u$.
For a finite set $A$ denote $\sharp A$ the cardinality of $A$.

\begin{lem}
\label{SmallAmountOfPoints}
Assume that all the closed points of $S^{\prime}$ have {\bf finite residue fields}.
Then there exists a finite \'{e}tale morphism
$\rho: S^{\prime\prime} \to S^{\prime}$ (with an irreducible scheme $S^{\prime\prime}$)
and a section
$\delta^{\prime}: T^{\prime} \to S^{\prime\prime}$
of $\rho$ over $T^{\prime}$ such that the following holds
\begin{itemize}
\item[\rm{(1)}]
for any closed point $u \in U$ let $u^{\prime} \in T^{\prime}$ be a unique point such that
$p(u^{\prime})=u$, then the point
$\delta^{\prime}(u^{\prime}) \in S^{\prime\prime}_u$ is the only
$k(u)$-rational point of $S^{\prime\prime}_u$,
\item[\rm{(2)}]
for any closed point $u \in U$ and any integer $r \geq 1$
one has
$$\sharp\{z \in S^{\prime\prime}_u| \text{deg}[z:u]=r \} \leq \ \sharp\{x \in \Aff^1_u| \text{deg}[z:u]=r \}$$
\end{itemize}
\end{lem}


\begin{thebibliography}{MMM}

\bibitem[A]{LNM305}
\emph{Artin, M.} Comparaison avec la cohomologie classique:
cas d'un pr\'esch\'ema lisse. Lect.~Notes Math., vol.~305, Exp XI.

\bibitem[C-T/O]{C-TO}
\emph{Colliot-Th\'el\`ene, J.-L.; Ojanguren, M.}
Espaces Principaux Homog\`enes Localement Triviaux,
Publ. Math. IHES 75, no.~2, 97--122 (1992).

\bibitem[C-T/P/S]{C-T/P/S}
\emph{Colliot-Th\'el\`ene, J.-L.; Parimala, R.; Sridaran, R.}
Un th\'eor\`eme de puret\'e local,
C. R. Acad. Sc. Paris 309, ser.I, 857--862 (1989).

\bibitem[C-T/S]{C-T-S}
\emph{Colliot-Th\'el\`ene, J.-L.; Sansuc, J.-J.}
Principal homogeneous spaces under flasque tori: Applications,
Journal of Algebra, 106 (1987), 148--205.

\bibitem[Ch]{Ch} V. Chernousov, Variations on a theme
of groups splitting by a quadratic extension and Grothendieck-Serre
conjecture for group schemes $F_4$ with trivial $g_3$ invariant,
Preprint Server: Linear Algebraic Groups and Related Structures, 354 (2009).

\bibitem[ChP]{ChP}
\emph{V.Chernousov, I.Panin} Purity of $G\sb 2$-torsors.
C.~R.~Math.~Acad.~Sci.~Paris 345 (2007), no. 6, 307--312.

\bibitem[D-G]{D-G}
\emph{Demazure, Grothendieck.} Structure des sch\'{e}mas en
groupes r\'eductifs. Lect.~Notes Math., vol 153.

\bibitem[E]{E}
\emph{Eisenbud, D.} Commutative Algebra,
Graduate Texts in Mathematics 150, Springer (1994).

\bibitem[FP]{FP} \emph{Fedorov, R.; Panin, I.} A proof of Grothendieck--Serre conjecture on principal bundles over a semilocal regular
ring containing an infinite field, Preprint, 2012,


\bibitem[Gr1]{Gr1}
\emph{Grothendieck, A.}
Torsion homologique et section rationnalles,
in {\it Anneaux de Chow et applications},
S\'{e}minaire Chevalley, 2-e ann\'{e}e, Secr\'{e}tariat math\'{e}matique,
Paris, 1958.

\bibitem[Gr2]{Gr2}
\emph{Grothendieck, A.}
Le group de Brauer II, in {\it Dix expos\'{e}s sur la cohomologique de sch\'{e}mes},
Amsterdam,North-Holland, 1968.

\bibitem[Gr3]{Gr3}
\emph{Grothendieck, A.}
Le group de Brauer III: Exemples et compl'ements, in {\it Dix expos\'{e}s sur la cohomologique de sch\'{e}mes},
Amsterdam,North-Holland, 1968.



\bibitem[KMRT]{KMRT}
\emph{Knus, M.A.; Merkurjev, A.S.; Rost, M.; Tignol, J.-P.},
The Book of Involutions, Colloquium Publication,
no.44, AMS, 1998.

\bibitem[H]{H}
\emph{Harder, G.}
Halbeinfache Gruppenschemata \"uber Dedekindringen,
Inventiones Math. 4, 165--191 (1967).

\bibitem[Ni]{Ni}
\emph{Nisnevich, Y.}
Rationally Trivial Principal Homogeneous Spaces and Arithmetic of
Reductive Group Schemes Over Dedekind Rings,
C. R. Acad. Sc. Paris, S\'erie I,
299, no.~1, 5--8 (1984).

\bibitem[Oj]{Oj}
\emph{Ojanguren, M.}
Quadratic forms over regular rings,
J. Indian Math. Soc., 44, 109--116 (1980).

\bibitem[OP]{OP}
\emph{Ojanguren, M; Panin, I.}
A purity theorem for the Witt group. \\
Ann. Sci. Ecole Norm. Sup. (4) 32,  no.~1, 71--86 (1999).

\bibitem[OP1]{OP1}
\emph{Ojanguren, M; Panin, I.}
Rationally trivial hermitian spaces are locally trivial.
Math.~Z. 237 (2001), 181--198.

\bibitem[OPZ]{OPZ}
\emph{Ojanguren, M; Panin, I.; Zainoulline, K.}
On the norm principle for quadratic forms.
J.~Ramanujan Math.~Soc. 19, No.4 (2004) 1--12.

\bibitem[PS]{PS}
\emph{Panin, I.; Suslin, A.}
On a conjecture of Grothendieck concerning Azumaya algebras,
St.Petersburg Math. J., 9, no.4, 851--858 (1998).

\bibitem[P2]{Pa2}
\emph{Panin, I.} On Grothendieck---Serre's conjecture concerning
principal $G$-bundles over reductive group schemes: II, Preprint,
April 2013,
{http://www.math.org/0905.1423v3}.


\bibitem[Pan1]{Pan1}
\emph{Panin, I.}
On Grothendieck-Serre conjecture
concerning principal $G$-bundles over regular semi-local domains containing a finite field: I,
Preprint,
May 2014.

\bibitem[Pan3]{Pan3}
\emph{Panin, I.} Proof of Grothendieck--Serre conjecture on principal $G$-bundles over semi-local regular
domains containing a finite field, Preprint, May 2014.


\bibitem[PSV]{PSV}
\emph{Panin, I.; Stavrova A,; Vavilov, N.}
On Grothendieck---Serre's conjecture concerning
principal $G$-bundles over reductive group schemes I,
Preprint, April, 2013,
{http://www.arxiv.org/0905.1418v3}


\bibitem[P]{P}
\emph{Popescu, D.}
General N\'eron desingularization and approximation,
Nagoya Math. Journal,
104 (1986), 85--115.


\bibitem[Pa]{Pa}
\emph{Panin, I.}
Purity for multipliers,
in Proceedings of the Silver Jubilee Conference University of Hyderabad
(Ed. R.Tandon) Algebra and Number Theory, pp.66--89

\bibitem[PaPS]{PaPS}
\emph{Panin, I.; Petrov, V.; Stavrova, A.}
On Grothendieck--Serre's for simple adjoint group schemes of types $E_6$ and $E_7$.
Preprint (2009),
{http://www.math.uiuc.edu/K-theory/}


\bibitem[R1]{R1}
\emph{Raghunathan, M.S.}
Principal bundles admitting a rational section. (English) Invent.
Math. 116, No.1--3, 409--423 (1994).

\bibitem[R2]{R2}
\emph{Raghunathan, M.S.}
Erratum: Principal bundles admitting a rational
section. (English) Invent. Math. 121, No.1, 223 (1995).

\bibitem[RR]{RR}
\emph{Raghunathan, M.S.; Ramanathan, A.}
Principal bundles on the affine line. (English)
Proc. Indian Acad. Sci., Math. Sci. 93, 137--145 (1984).

\bibitem[R]{R}
\emph{Rost, M.}
Durch Normengruppen definierte birationale Invarianten.
C. R. Acad. Sc. Paris 10, ser.I, 189--192 (1990).

\bibitem[Q]{Q}
\emph{Quillen, D.}
Higher algebraic $K$-theory I, in Algebraic K-Theory I.\\
Lect. Notes in Math. 341, Springer-Verlag, 1973.

\bibitem[Se]{Se}
\emph{Serre, J.-P.}
Espaces fibr\'{e}s alg\'{e}briques, in
{\it Anneaux de Chow et applications},
S\'{e}minaire Chevalley, 2-e ann\'{e}e, Secr\'{e}tariat math\'{e}matique,
Paris, 1958.


\newblock
\bibitem[SuVo]{SV}
\emph{Suslin, A.; Voevodsky, V.}
Singular homology of abstract algebraic varieties.
Invent. Math. 123 (1996), no. 1, 61--94.


\bibitem[Sw]{Sw}
\emph{Swan, R.G.}
N\'eron-Popescu desingularization,
Algebra and Geometry (Taipei, 1995),
Lect. Algebra Geom. Vol.\ 2,
Internat. Press, Cambridge, MA, 1998, pp.~135--192.

\bibitem[Z]{Z}
\emph{Zainoulline, K.}
The Purity Problem for Functors with Transfers,
K-theory, 22 (2001), no.4, 303-333.

\bibitem[Vo]{Vo}
\emph{Voevodsky, V.}
Cohomological theory of presheaves with transfers. In
Cycles, Transfers, and Motivic Homology Theories.
Ann.~Math.~Studies, 2000, Princeton University Pres.

\end{thebibliography}
\end{document}